\documentclass[12pt, reqno]{amsart}
\usepackage[english]{babel}
\usepackage[utf8]{inputenc}
\usepackage{lipsum}
\usepackage{dsfont}
\usepackage{mathrsfs}
\usepackage{stix}
\usepackage{fullpage}
\usepackage{mathtools}
\usepackage{amsmath}
\usepackage{bbm}
\usepackage{amsfonts}
\usepackage{tikz}
\usepackage{tikz-cd}
\usetikzlibrary{decorations.pathreplacing,angles,quotes}
\usetikzlibrary{arrows,chains,positioning,scopes,quotes}
\newtheorem{defi}{Definition}[section]

\newtheorem*{conjecture}{Conjecture}
\newtheorem{lema}[defi]{Lemma}
\newtheorem{teo}[defi]{Theorem}
\newtheorem{rem}[defi]{Remark}

\newtheorem*{rem*}{Remark}

\newcommand{\C}{\mathbb{C}}
\newcommand{\Q}{\mathbb{Q}}
\newcommand{\K}{\mathbb{K}}
\newcommand{\R}{\mathbb{R}}
\newcommand{\N}{\mathbb{N}}
\newcommand{\Z}{\mathbb{Z}}

\newcommand{\1}{\mathds{1}}

\DeclareMathOperator{\spn}{span}

\newcommand{\interior}[1]{%
  {\kern0pt#1}^{\mathrm{o}}%
}

\newcommand{\esp}{\text{  }}

\usepackage[colorlinks=false]{hyperref}
\renewcommand\eqref[1]{(\ref{#1})} 
\begin{document}

\title[The spectrum of the Vladimirov sub-Laplacian on the compact Engel group ]
 {The spectrum of the Vladimirov sub-Laplacian on the compact Engel group }

\author{
  J.P. Velasquez-Rodriguez
}

\newcommand{\Addresses}{{
  \bigskip 
  \footnotesize

 J.P. Velasquez-Rodriguez, \textsc{ Departamento de Matematicas, Universidad del Valle, Cali-Colombia}\par\nopagebreak
  \textit{E-mail address:} \texttt{ velasquez.juan@correounivalle.edu.co}}

}


\thanks{The author is supported by the grace of God, a.k.a. The Truth.}



\date{\today}
\begin{abstract}
Let $p>3$ be a prime number. In this note, we use $p$-adic Gaussian integrals to calculate explicitly the unitary dual and the matrix coefficients of the Engel group over the $p$-adic integers $\mathcal{B}_4(\Z_p)$. We use this information to calculate explicitly the spectrum of the Vladimirov sub-Laplacian, and show how it defines a globally hypoelliptic operator on $\mathcal{B}_4$.  
\end{abstract}
\maketitle
\tableofcontents
\section{Introduction}
\subsection{Noncommutative $p$-adic analysis} The publication of Joseph Fourier's book, Théorie Analytique de la Chaleur \cite{Fourier2009}, is an important landmark in the history of mathematics and physics, as it contains the first extended mathematical treatment of the solutions to the heat diffusion equation $$\frac{\partial f}{\partial t} =  -\mathscr{L} f, \quad \mathscr{L}:= d^2/dx^2,$$ where $\mathscr{L}$ denotes the one-dimensional Laplacian, in terms of certain special series and integrals. Fourier's idea actually turned out to be a very potent one: we can represent real valued functions, interpreted as heat distributions on a certain object, as sums of simpler functions: sines and cosines, which turn out to be the eigen-functions of $\mathscr{L}$. The natural extension of this idea is to do the same for complex-valued functions, which can be represented as sums of complex exponential functions, denoted as $\mathscr{e}_k(\theta) = e^{2 \pi i k \theta}$, where $k \in \mathbb{Z}$. This way of expressing functions is a particular case of a much broader phenomenon that would only be fully appreciated centuries later, thanks to the advancements in group theory and topological vector spaces. From an algebraic perspective, the set $\widehat{\mathbb{T}} := \{\mathscr{e}_k : k \in \mathbb{Z}\}$ serves as a complete collection of (representatives of equivalence classes of) unitary irreducible representations of the compact group $\mathbb{T} := \{ z \in \mathbb{C} : z z^* = 1 \}$, where $z^*$ denotes the complex conjugate of $z$. And from the perspective of topological vector spaces, the Peter-Weyl theorem offers a more comprehensive understanding of the phenomenon Fourier identified, demonstrating how $\widehat{\mathbb{T}}$ forms an orthonormal basis for $L^2(\mathbb{T})$.

The generalizations of Fourier ideas are usually known as Fourier analysis, and we can mention many examples of topological spaces where there is a way to decompose functions as sums of some other functions "simpler to understand". If one sticks to spaces with an additional algebraic structure, which are understood as spaces rich in symmetries, there are two different directions one can take:    
\begin{enumerate}
    \item \textbf{To study locally connected spaces}: the natural generalization of the theory on $\mathbb{T}$ involves considering compact Lie groups. This concept is systematically addressed in \cite{Ruzhansky2010}, which outlines the foundational elements of a general theory of pseudo-differential operators on compact Lie groups. A key ingredient, though often a significant challenge, is understanding the representation theory of the group in question, which can be quite complex depending on the specific case. For instance, see \cite{Ruzhansky2010} for the case of $\mathrm{SU}(2)$ and \cite{Cerejeiras2023} for spin groups.
    \item \textbf{To study t.d. spaces}: In the realm of totally disconnected spaces, a rich theory exists concerning Fourier analysis on non-Archimedean local fields and their rings of integers. The first generalization in this context is to examine the so-called \emph{locally compact abelian Vilenkin groups}, which are topological groups characterized by an specific sequence of compact open subgroups. See for instance \cite{pseudosvinlekinsaloff}. Analogous to the locally connected case, when transitioning to noncommutative groups, one must consider compact $p$-adic Lie groups, which have primarily been explored from an algebraic perspective in the literature \cite{Boyarchenko2008, Howe1977KirillovTF}, as well as the broader category of \emph{compact Vilenkin groups}.
\end{enumerate}

Here in this paper we are taking the second direction, starting with the exploration of compact nilpotent $p$-adic Lie groups. Over the past 15 years, a comprehensive symbol calculus for compact (noncommutative) groups has been developed by Ruzhansky, Turunen, and Wirth, even thought their main focus has been compact Lie groups over $\R$ \cite{Ruzhansky2010}. This new framework serves as a noncommutative extension of the classical Kohn–Nirenberg quantization, which is a tool for studying differential operators on the group, and offers several advantages over Hörmander's principal calculus or any other local methods. For a given Lie group \( G \), this approach leverages its representation theory and the corresponding harmonic analysis to establish a global Fourier transform which diagonalizes differential operators, including the Laplace-Beltrami operator, which is actually a far reaching generalization of Fourier's original idea. In other words, we already have the basic elements of Fourier analysis in general compact groups but, to the best knowledge of the author, there is no literature about something like the matrix coefficients of compact $p$-adic Lie groups, or quantization of linear operators in such setting. We intend to contribute here to such gap by describing in full detail the unitary dual of one particular example: the \emph{Engel group over the $p$-adic integers} $\mathcal{B}_4 (\Z_p)$, or just $\mathcal{B}_4$ for simplicity. It can be defined as $\Z_p^4$ endowed with the noncommutative operation $$\mathbf{x}\star \mathbf{y}:=(x_1 + y_1, x_2 + y_2, x_3 + y_3 + x_1 y_2, x_4 + y_4 + x_1 y_3 + \frac{1}{2} x_1^2 y_2),$$so that it is an analytic manifold over $\Q_p$, whose tangent space at the identity can be identified with $\mathfrak{b}_4$, the $4$-dimensional Engel algebra, that is, the $\Z_p$-Lie algebra generated by $X_1,..,X_4$, with the commutation relations $$[X_1 , X_2] = X_3, \,\,\, [X_1 , X_3 ] = X_4, \,\,\, \mathfrak{b}_4 = \bigoplus_{j=1}^3 \mathcal{V}_j, $$where $$\mathcal{V}_1 = \mathrm{span}_{\Z_p} \{ X_1 , X_2 \}, \,\,\, \mathcal{V}_1 = \mathrm{span}_{\Z_p} \{ X_3 \}, \,\,\, \mathcal{V}_3 = \mathrm{span}_{\Z_p} \{ X_4 \}.$$
The above defined Lie algebra has been considered over the field of real numbers by many authors. Here we are interested in the perspective adopted in \cite{Boscain2014}, which is actually similar to \cite{Rockland1978} for the Heisenberg group, where the group Fourier analysis is used to study the heat kernel of the sub-Laplacian on $\mathcal{B}_4(\R)$. In that case the group is not compact, and the representations need to be infinite-dimensional, but still one can simplify problems like the hypoelliptic heat equation by transforming it to a system of equations in each representation space.  

In this paper, we are interested in the same problem, but we replaced the field $\R$ for the compact ring $\Z_p$. Just as for the real case, the possibility of understanding properly the Fourier analysis depends on our knowledge of the relevant representation theory, but in this case the group $\mathcal{B}_4(\Z_p)$ is compact, so our representation spaces are all finite-dimensional and the theory is more in the spirit of \cite{Ruzhansky2010}. Of course, complex representations of $p$-adic groups like $\mathcal{B}_4(\Z_p)$ have been considered before, see for instance \cite{Boyarchenko2008, Howe1977KirillovTF} for the general theory. However, even though there are some versions of Fourier analysis on compact groups using the characters of the representations, for many purposes having the matrix coefficients is crucial, and this seem to be an information missing in the mathematical literature. We intend to contribute to this gap here by providing explicit formulas for the matrix coefficients of $\mathcal{B}_4(\Z_p)$, and to used them to solve the problem of the global hypoellipticity of the \emph{Vladimirov sub-Laplacian}. With this, and the calculations in \cite{SublaplacianHeisenberg}, the cases of all compact nilpotent $p$-adic Lie groups of dimension up to $4$ are covered, and the recent paper \cite{MatrixCoeff5d} addresses the case od dimension $d=5$.   

As it is well known, in the $p$-adic side of the theory the usual notion of derivative is not available, and neither infinitesimal representations. However, there is a way to define what it means for a function to be "differentiable" called the Vladimirov-Taibleson operator \cite{Khrennikov2018, Taibleson2015-ov, vladiBook}. The notion we want to introduce here is that, despite not having infinitesimal representations, we can always use this way of defining differentiability to assign a certain operator to each element of the Lie algebra of a $p$-adic Lie group. This assignation does not preserve the Lie algebra structure, because there is actually no meaningful way to define complex representations of a $p$-adic Lie algebra, but still we can use these operators to study some interesting problems, like the heat equation or the hypoelliptic heat equation. This development aims to contribute with the growing body of literature about $p$-adic analysis \cite{Dragovich2017}, and to open the door to investigations about the Vladimirov-type operators on $p$-adic manifolds with an additional algebraic structure.      

\begin{defi}\normalfont\label{defiDirectionalK}
Let $\K$ be a non-archimedean local field with ring of integers $\mathscr{O}_\K$, prime ideal $\mathfrak{p}= \textbf{p} \mathscr{O}_\K$, and residue field $\mathbb{F}_q = \mathscr{O}_\K/\textbf{p} \mathscr{O}_\K.$ Let $\mathfrak{g} = \spn_{\mathscr{O}_\K} \{ X_1,..,X_d \}$ be a nilpotent $\mathscr{O}_\K$-Lie algebra, and let $\mathbb{G}$ be the exponential image of $\mathfrak{g}$, so that $\mathbb{G}$ is a compact nilpotent $\K$-Lie group. We will use the symbol $\partial_{X}^\alpha$ to denote the \emph{directional Vladimirov–Taibleson operator in the direction of $X \in \mathfrak{g}$}, or directional VT operator for short, which we define as $$\partial_{X}^\alpha f(\mathbf{x}) := \frac{1 - q^\alpha}{1-q^{-(\alpha +1)}} \int_{\mathscr{O}_\K} \frac{f(\mathbf{x} \star \mathbf{exp}(tX)^{-1}) - f(\mathbf{x})}{|t|_\K^{\alpha +1}}dt, \, \, \, \, \alpha>0, \, \,  f \in \mathcal{D}(\mathbb{G}).$$Here $\mathcal{D}(\mathbb{G})$ denotes the space of smooth functions on $\mathbb{G}$, i.e., the collection of locally constant functions with a fixed index of local constancy.   
\end{defi}

Notice how the above definition is independent of any coordinate system we choose for the group as a manifold. Even though we are using the exponential map in the definition, this does not mean we are using the exponential system of coordinates. First, $\mathbf{x} \in \mathbb{G}$ is not written in coordinates, and second, we know there is a one to one correspondence between elements of the Lie algebra $\mathfrak{g}$ and one-parameter subgroups of $\mathbb{G}$. Here a one-parameter subgroup of $\mathbb{G}$ is an analytic homomorphism $\gamma_X:\mathscr{O}_\K \to \mathbb{G}$.  Using this fact, given any $X \in \mathfrak{g}$, we can define the directional VT operator in the direction of $X \in \mathfrak{g}$ via the formula  $$\partial_{X}^\alpha f(\mathbf{x}) := \frac{1 - q^\alpha}{1-q^{-(\alpha +1)}} \int_{\mathscr{O}_\K} \frac{f(\mathbf{x} \star \gamma_X(t)^{-1}) - f(x) }{|t|_\K^{\alpha +1}}dt.$$
Since $\mathbb{G}$ is compact, we can identify it with a matrix group, where all one parameter subgroups have the form $\gamma_X (t) = e^{tX}$, which justifies our initial definition. More generally, we can write any analytic vector field on $\mathbb{G}$ as $$X(\mathbf{x}):= 
\sum_{j=1}^d c_j (\mathbf{x}) X_j,$$where the coefficient functions $c_j$ are analytic. For this vector field, the directional VT operator is defined in a similar way as $$\partial_{X}^\alpha f(\mathbf{x}) := \frac{1 - q^\alpha}{1-q^{-(\alpha +1)}} \int_{\mathscr{O}_\K} \frac{f(\mathbf{x} \star \mathbf{exp}(tX (\mathbf{x}))^{-1}) - f(\mathbf{x}) }{|t|_\K^{\alpha +1}}dt.$$However, in this work we wont consider the case where the coefficients are not constant, because such operators are not necessarily invariant.         

To the best of the author's knowledge, operators like the ones introduced in Definition \ref{defcompactdirectionalVT} have not been studied before in the literature, except maybe for the case of the \emph{Vladimirov Laplacian} studied in \cite{Bendikov2014, Rajkumar2023}. Still, many operators similar to these have been considered in the literature, even though only on local fields or abelian groups. So far, in the knowledge of the author, there is no theory of derivative, differential, or pseudo-differential operators on $p$-adic manifolds, and neither on the class of noncommutative locally profinite groups, even though there is a rich theory on $p$-adic Lie groups, algebraic groups, and $\K$-analytic manifolds. Still there are some studies about pseudo-differential operators invariant under a finite group actions, heat equations on Mumford curves, and $p$-adic Laplacians \cite{Bradley3, Bradley2, bradley1}, and heat equations on a $p$-adic ball \cite{Kochubei2018}, which motivates the generalizations to the noncommutative case that we want to treat here. 

\subsection{The problem of global hypoellipticity}

When considering a partial differential operator \( L \), a fundamental question in PDE theory is: \emph{If \( Lu = f \) and \( f \) is smooth, does this imply that \( u \) is also smooth?} This question, which plays a central role in the theory of pseudo-differential operators on smooth manifolds, introduces the concept of \emph{hypoellipticity}, see for instance \cite{Bramanti2014} and related references. For example, in 1967, H{\"o}rmander demonstrated that, under the assumption that the system of smooth vector fields \( X_1, \ldots, X_\kappa \) spans the entire tangent space at every point, a condition now known as \emph{H{\"o}rmander's condition}, the operator
\begin{equation}
    L := \sum_{i=1}^\kappa X_i^2 + X_0 + c
\end{equation}
is hypoelliptic \cite{Hrmander1967}. Such operators are commonly referred to as \emph{H{\"o}rmander's operators}. Special cases include the \emph{H{\"o}rmander's} sum of squares when \( X_0 = 0 \) and \( c = 0 \), and important examples of H{\"o}rmander's operators include sub-Laplacians on Lie groups \cite{Bramanti2014}.

This paper addresses the problem of global hypoellipticity on compact nilpotent groups over non-archimedean local fields. Specifically, we introduced before the concept of \emph{directional Vladimirov-Taibleson operators}, and it is our goal now to understand the behavior of functions of these operators, with particular emphasis on polynomial functions. These operators provide a framework for understanding differentiability on functions defined over profinite groups, with directional VT operators bearing similarities to directional derivatives. By using Definition \ref{defcompactdirectionalVT}, we can link an specific Vladimirov-type operator to each direction $X \in \mathfrak{g}$ and subsequently study the resulting operators. While this association does not preserve the Lie algebra structure, as seen in Lie groups over the real numbers, the resulting operators are nonetheless interesting and share similarities with partial differential operators on Lie groups. We want to study polynomials in these operators and, specifically, a distinguished operator called here the \emph{Vladimirov sub-Laplacian}. For this operator, which we define in principle for compact stratified groups, we would like to pose the following conjecture:

\begin{conjecture}\normalfont   
 Let $\K$ be a non-archimedean local field with ring of integers $\mathscr{O}_\K$, prime ideal $\mathfrak{p}= \textbf{p} \mathscr{O}_\K$, and residue field $\mathbb{F}_q = \mathscr{O}_\K/\textbf{p} \mathscr{O}_\K.$ Let $\mathfrak{g} = \spn_{\mathscr{O}_\K} \{ X_1,..,X_d \}$ be a nilpotent $\mathscr{O}_\K$-Lie algebra, and let $\mathbb{G}$ be the exponential image of $\mathfrak{g}$, so that $\mathbb{G}$ is a compact nilpotent $\K$-Lie group. Let $X_1,...,X_{\kappa}$, $1 \leq \kappa \leq d$, be a basis for $\mathfrak{g}/[\mathfrak{g},\mathfrak{g}]$, so that $X_1,...,X_{\kappa}$ generates $\mathfrak{g}$. Then the \emph{Vladimirov sub-Laplacian of order $\alpha>0$} is a hypoelliptic operator on $\mathbb{G}$, and it is invertible in the space of mean-zero functions. Here the Vladimirov sub-Laplacian is the pseudo-differential operator $\mathscr{L}^\alpha_{sub}$, defined on the space of smooth functions $\mathcal{D}(\mathbb{G})$ via the formula $$\mathscr{L}^\alpha_{sub} f(\mathbf{x}) := \sum_{k= 1}^\kappa 
\partial_{X_k}^\alpha f(\mathbf{x}),$$where $X_1 , ...,X_\kappa $ spans $\mathfrak{g}/[\mathfrak{g},\mathfrak{g}]$. 
\end{conjecture}\begin{rem}\normalfont
    To the knowledge of the author, an operator like the Vladimirov sub-Laplacian has not appeared before in the mathematical literature. The closest thing is probably the study of the \emph{Vladimirov-Laplacian} by Bendikov,  Grigoryan,  Pittet,  and Woess \cite{Bendikov2014}. There the authors study the operators $$\mathscr{L}^\alpha f(x) := \sum_{k= 1}^d 
\partial_{X_k}^\alpha f(x),$$as the generators of a Dirichlet form associated with a certain jump kernel. See \cite[Section 5]{Bendikov2014} for more details, and also the operators studied in \cite{Rajkumar2023}. 
\end{rem}

In the particular case of $\mathfrak{g}=\mathfrak{b}_4$, we can see how this algebra is stratified because the pair $\{ X_1 , X_2 \}$ generates the whole algebra, which is actually H{\"o}rmanders condition for a Lie group, so its associated \emph{Vladimirov sub-Laplacian} is going to be:
$$\mathscr{L}^\alpha_{sub} f(\mathbf{x}):=(\partial_{X_1}^\alpha + \partial_{X_2}^\alpha) f(\mathbf{x}).$$Drawing an analogy between the real case and the non-archimedean case, we can think on the Vladimirov sub-Laplacian as an analog of the sub-Laplacian or H{\"o}rmander's sum of squares for the non-archimedean case. Consequently, it is reasonable to expect the global hypoellipticity for these operators, and one of the main goals of this paper is to demonstrate that this is indeed the case for the $4$-dimensional Engel group over the $p$-adic integers. See for instance \cite[Example 2.2.7]{Corwin2004-kz} for the definition of the Engel algebra and the Engel group over the reals.
 
 \begin{rem}
    In this paper, we will identify each equivalence class $\lambda$ in $\widehat{\Z}_p \cong \Q_p / \Z_p$, with its associated representative in the complete system of representatives $$\{1\} \cup \big\{ \sum_{k =1}^\infty \lambda_k p^{-k} \, : \, \, \text{only finitely many $\lambda_k$ are non-zero.} \big\}.$$Similarly, every time we consider an element of the quotients $\Q_p / p^{-n} \Z_p$ it will be chosen from the complete system of representatives   $$\{1\} \cup \big\{ \sum_{k =n+1}^\infty \lambda_k p^{-k} \, : \, \text{only finitely many $\lambda_k$ are non-zero} \big\}.$$For any $p$-adic number $\lambda$, we will use the symbol $\vartheta(\lambda)$ to denote its corresponding $p$-adic valuation. 
\end{rem}
The convention introduced above is important because it is used in the description of the unitary dual of the Heisenberg group over the $p$-adic integers, exposed in detail in \cite{SublaplacianHeisenberg}. We recall it here because we will use it to simplify the description of the unitary dual of $\mathcal{B}_4$. If we denote by $\mathcal{Z}(\mathcal{B}_4)$ the center of $\mathcal{B}_4(\Z_p)$, then we can say there are two different kind of representations: those which are trivial on the center and the one that are not. If an unitary irreducible representation of $\mathcal{B}_4(\Z_p)$ is trivial on the center, then it must be one of the representations of $\mathcal{B}_4(\Z_p) / \mathcal{Z}(\mathcal{B}_4)\cong \mathbb{H}_1 (\Z_p)$, and these where calculated in \cite{SublaplacianHeisenberg}. Because of this, part of the work is already done and here we only need to focus on those representations which act non-trivially on the center.  
 
\begin{teo}\label{TeoRepresentationsHd}
Let $\mathbb{H}_1 (\Z_p)$, or simply $\mathbb{H}_1$ for short, be the $3$-dimensional Heisenberg group over the $p$-adic integers. Let us denote by $\widehat{\mathbb{H}}_1$ the unitary dual of $\mathbb{H}_1$, i.e., the collection of all equivalence classes of unitary irreducible representations of $\mathbb{H}_1$. Then we can identify $\widehat{\mathbb{H}}_1$ with the following subset of $\widehat{\Z}_p^{3} \cong \Q_p^{3}/\Z_p^{3}$: $$\widehat{\mathbb{H}}_1 = \{(\xi_1 ,\xi_2 , \xi_3 ) \in  \widehat{\Z}_p^{3} \, : \, (\xi_1, \xi_2) \in \Q_p^2/ p^{\vartheta(\xi_3)} \Z_p^2 \}.$$Moreover, each non-trivial representation $[\pi_{(\xi_1, \xi_2, \xi_3)}] \in \widehat{\mathbb{H}}_1 $ can be realized in the finite dimensional sub-space $\mathcal{H}_{\xi_3}$ of $L^2(\Z_p)$ defined as
$$\mathcal{H}_{\xi_3} := \spn_\C \{ \varphi_h \, : \, h \in \Z_p / p^{-\vartheta(\xi_3)} \Z_p  \}, \, \, \, \varphi_h (u) := |\xi_3|_p^{1/2} \1_{h + p^{-\vartheta(\xi_3)} \Z_p} (u), \, \, \dim_\C(\mathcal{H}_{\xi_3}) =|\xi_3|_p,$$
where the representation acts on functions $\varphi \in \mathcal{H}_{\xi_3}$ according to the formula 
$$\pi_{(\xi_1 , \xi_2 , \xi_3)}(\mathbf{x}) \varphi (u) := e^{2 \pi i \{\xi_1 x_1 + \xi_2 x_2 + \xi_3 (x_3 +  u x_2) \}_p} \varphi (u + x_1), \, \, \, \, \varphi \in \mathcal{H}_{\xi_3}.$$With this explicit realization, and by choosing the basis $\{ \varphi_h\,: \, h \in \Z_p / p^{-\vartheta(\xi_3)}\Z_p\}$ for each representation space, the associated matrix coefficients are given by $$(\pi_{(\xi_1 , \xi_2 , \xi_3)})_{h h'}(\mathbf{x})=e^{ 2 \pi i \{ x_1 \xi_1 + x_2 \xi_2 + \xi_3(x_3 + h' x_2 )  \}_p} \1_{h' - h + p^{-\vartheta(\xi_3)} \Z_p}(x_1) .$$
\end{teo}

As the main results of this paper, we will provide a complete description of the unitary dual, the matrix coefficients of the representations, and the spectrum of the Vladimirov sub-Laplacian on $\mathcal{B}_4(\Z_p)$. To be more precise, our goal in this work is to establish the following two theorems: 

\begin{teo}\label{TeoRepresentationsB4}
Let $\mathcal{B}_4(\Z_p)$, or simply $\mathcal{B}_4$ for short, be the $4$-dimensional Engel group over the $p$-adic integers. Let us denote by $\widehat{\mathcal{B}}_4$ the unitary dual of $\mathcal{B}_4$, i.e., the collection of all unitary irreducible representations of $\mathcal{B}_4$. Then we can identify $\widehat{\mathcal{B}}_4$ with the following subset of $\widehat{\Z}_p^{4} \cong \Q_p^{4}/\Z_p^{4}$:  
$$\widehat{\mathcal{B}}_4:= \{ \xi  \in \widehat{\Z}_p^4 \, : \, 1 \leq |\xi_4|_p <|\xi_3|_p \, \wedge \,  (\xi_1 , \xi_2 , \xi_3) \in \widehat{\mathbb{H}}_1 , \, \text{or} \,  |\xi_3|_p = 1 \,  \wedge  \,  \xi_1 \in \Q_p / p^{\vartheta(\xi_4)} \Z_p \}.$$
Moreover, each non-trivial representation $[\pi_\xi]$ is equivalent to one of the representation $[ \pi_\xi ] \in \widehat{\mathcal{B}}_4 $ which can be realized in the following finite dimensional sub-space $\mathcal{H}_\xi$ of $L^2(\Z_p)$:
\[ \mathcal{H}_\xi :=  \mathrm{span}_\C \{ \|(\xi_3, \xi_4)\|_p^{1/2} \1_{h + p^{-\vartheta(\xi_3, \xi_4)} \Z_p } \,\, : \, h \in \Z_p / p^{-\vartheta(\xi_3, \xi_4)} \Z_p \},    \]
where $d_\xi:= \dim_\C (\mathcal{H}_\xi) = \max\{ |\xi_3|_p , |\xi_4|_p \}$, and the representation acts on functions $\varphi \in \mathcal{H}_{\xi}$ according to the formula 
\[\pi_{\xi}(\mathbf{x}) \varphi (u) := 
    e^{2 \pi i \{\xi_1 x_1 + \xi_2 x_2 + \xi_3 (x_3 +  u x_2)  + \xi_4 (x_4 +  u x_3 + \frac{u^2}{2} x_2) \}_p} \varphi (u + x_1)  . \]
With this explicit realization of $[\pi_\xi]  \in \widehat{\mathcal{B}}_4 $, and with the natural choice of basis for each representation space, the associated matrix coefficients $(\pi_{\xi})_{h h'}$ are given by \[(\pi_{\xi}(\mathbf{x}))_{hh'} = e^{2 \pi i \{\xi \cdot x +(\xi_3   x_2 +  \xi_4 x_3)(h') +  \xi_4 x_2 \frac{(h')^2}{2} \}_p} \1_{h' - h + p^{-\vartheta(\xi_3 , \xi_4)} \Z_p} (x_1) , \]and their associated characters are going to have the form 
\begin{align*}
    \chi_{\pi_\xi} (\mathbf{x}) = \|(\xi_3, \xi_4)\|_p  e^{2 \pi i \{\xi \cdot \mathbf{x} \}_p} \1_{p^{- \vartheta(\xi_3, \xi_4) } \Z_p} (x_1) \int_{ \Z_p}e^{2 \pi i \{(\xi_3   x_2 +  \xi_4 x_3)u +  \xi_4 x_2 \frac{u^2}{2} \}_p}  du.
\end{align*}
Sometimes we will use the notation $$\mathcal{V}_{\xi}:= \mathrm{span}_\C \{ (\pi_{\xi})_{h h'} \, : \, h,h' \in \Z_p/p^{-\vartheta(\xi_3 , \xi_4)} \Z_p \}, $$
\end{teo}
With the description of the unitary dual given in Theorem \ref{TeoRepresentationsB4}, we can prove the following spectral theorem for the Vladimirov sub-Laplacian on $\mathcal{B}_4$. This will provide a proof of our conjecture up to dimension $4$, and it is shown in \cite{MatrixCoeff5d} how indeed it holds for dimension $d=5$ too.  

\begin{teo}\label{TeoSpectrumSublaplacianB4}
The Vladimirov sub-Laplacian associated to the basis $\{X_1 ,X_2\}$ of the first stratum of $\mathfrak{b}_4$ $$\mathscr{L}_{sub}^{\alpha } : =  \partial_{X_1}^{\alpha} + \partial_{X_2}^{\alpha} ,$$defines a left-invariant, self-adjoint, sub-elliptic operator on $\mathcal{B}_4$. The spectrum of this operator is purely punctual, and its associated eigenfunctions form an orthonormal basis of $L^2(\mathcal{B}_4)$. Furthermore, the symbol of $\mathscr{L}_{sub}^{\alpha }$ acts on each representation space as a $p$-adic Schr{\"o}dinger type operator, and the space $L^2(\mathcal{B}_4)$ can be written as the direct sum $$L^2(\mathcal{B}_4) = \overline{\bigoplus_{[\xi] \in \widehat{\mathcal{B}}_4} \bigoplus_{h' \in  \Z_p / p^{-\vartheta(\xi_3 , \xi_4)} \Z_p} \mathcal{V}_{\xi}^{h'}}, \, \,\, \mathcal{V}_{\xi} = \bigoplus_{{h'} \in  \Z_p / p^{-\vartheta(\xi_3 , \xi_4)} \Z_p} \mathcal{V}_{\xi}^{h'}, $$where each finite-dimensional sub-space$$\mathcal{V}_{\xi}^{h'}:= \mathrm{span}_\C \{ (\pi_{\xi})_{hh'} \, : \, h \in \Z_p / p^{-\vartheta(\xi_3 , \xi_4)} \Z_p \},$$is an invariant sub-space where $\mathscr{L}_{sub}^{\alpha }$ acts like the Schr{\"o}dinger-type operator operator $$  \partial_{x_1}^{\alpha} + Q(\xi , h') - \frac{1 - p^{-1}}{1 - p^{-(\alpha +1)}}.$$
Consequently, the spectrum of $\mathscr{L}_{sub}^{\alpha }$ restricted to $\mathcal{V}_{\xi}^{h'}$ is purely punctual, and it is given by \[\begin{cases} |\xi_1|_p^\alpha + |\xi_2|^\alpha_p - 2\frac{1 - p^{-1}}{1 - p^{-(\alpha +1)}} \quad & \text{if} \quad |\xi_3|_p = |\xi_4 |_p =1,   \\ |\xi_1 + \tau|_p^{\alpha} +| \xi_2  + \xi_3    h'  |_p^\alpha - 2\frac{1 - p^{-1}}{1 - p^{-(\alpha +1)}} \quad & \text{if} \quad 1= |\xi_4|_p <|\xi_3|_p,  \\|\xi_1 + \tau|_p^{\alpha} +| \xi_2  + \xi_3    h' +  \xi_4  \frac{(h')^2}{2} |_p^\alpha - 2\frac{1 - p^{-1}}{1 - p^{-(\alpha +1)}} \quad & \text{if} \quad 1< |\xi_4|_p <|\xi_3|_p, \\ |\xi_1 + \tau|_p^{\alpha} +| \xi_2  + \xi_4  \frac{(h')^2}{2} |_p^\alpha - 2\frac{1 - p^{-1}}{1 - p^{-(\alpha +1)}} \quad & \text{if} \quad 1= |\xi_3|_p <|\xi_4|_p, 
\end{cases} \]where $\tau \in \Z_p / \| (\xi_3 , \xi_4)\|_p \Z_p$, and the corresponding eigenfunctions are given by  $$\mathscr{e}_{\xi, h', \tau}(\mathbf{x}): = e^{ 2 \pi i \{ \xi \cdot \mathbf{x}  + \tau x_1 + (\xi_3   x_2 +  \xi_4 x_3)(h') +  \xi_4 x_2 \frac{(h')^2}{2} \}_p} ,$$ 
where $$\xi \in \widehat{\mathcal{B}}_4, \, h '\in \Z_p / p^{-\vartheta(\xi_3 , \xi_4)} \Z_p, \,  1 \leq | \tau |_p \leq \| (\xi_3 , \xi_4)\|_p .$$
\end{teo}

\section{Preliminaries}
\subsection{The field of $p$-adic numbers $\Q_p$}
Throughout this article $p>2$ will denote a fixed prime number. The field of $p$-adic numbers, usually denoted by $\Q_p$, can be defined as the complection of the field of rational numbers $\Q$ with respect to the $p$-adic norm $|\cdot|_p$, defined as \[|u|_p := \begin{cases}
0 & \esp \text{if} \esp u=0, \\ p^{-l} & \esp \text{if} \esp u= p^{\gamma} \frac{a}{b},
\end{cases}\]where $a$ and $b$ are integers coprime with $p$. The integer $l:= \vartheta(u)$, with $\vartheta(0) := + \infty$, is called the \emph{$p$-adic valuation of $u$}. The unit ball of $\Q_p^d$ with the $p$-adic norm $$\|u \|_p:=\max_{1 \leq j \leq d} |u_j|_p = \max_{1 \leq j \leq d} p^{-\vartheta(u_j)} =:p^{- \vartheta(u)}, \quad \vartheta(u):= \min_{1 \leq j \leq d} \vartheta(u_j),$$ is called the group of $p$-adic integers, and it will be denoted by $\Z_p^d$. Any $p$-adic number $u \neq 0$ has a unique expansion of the form $$u = p^{\vartheta(u)} \sum_{j=0}^{\infty} u_j p^j,$$where $u_j \in \{0,1,...,p-1\}$ and $u_0 \neq 0$. With this expansion we define the fractional part of $u \in \Q_p$, denoted by $\{u\}_p$, as the rational number\[\{u\}_p := \begin{cases}
0 & \esp \text{if} \esp u=0 \esp \text{or} \esp \vartheta(u) \geq 0, \\ p^{\vartheta(u)} \sum_{j=0}^{-\vartheta-1} u_j p^j,& \esp \text{if} \esp \vartheta(u) <0.
\end{cases}\] $\Z_p^d$ is  compact, totally disconected, i.e. profinite, and abelian. Its dual group in the sense of Pontryagin, the collection of characters of $\Z_p^d$ will be denoted by $\widehat{\Z}_p^d$. The dual group of the $p$-adic integers is known to be the Pr{\"u}fer group $\Z (p^{\infty})$,  the unique $p$-group in which every element has $p$ different $p$-th roots. The Pr{\"u}fer group may be identified with the quotient group $\Q_p/\Z_p$. In this way the characteres of the group $\Z_p^d$ may be written as $$\chi_p (\tau  u) := e^{2 \pi i \{\tau \cdot u \}_p}, \esp \esp u \in \Z_p^d, \esp \tau \in \widehat{\Z}_p^d 
\cong \Q_p^d / \Z_p^d .$$

 \begin{rem}
    In this paper, we will identify each equivalence class $\lambda$ in $\widehat{\Z}_p \cong \Q_p / \Z_p$, with its associated representative in the complete system of representatives $$\{1\} \cup \big\{ \sum_{k =1}^\infty \lambda_k p^{-k} \, : \, \, \text{only finitely many $\lambda_k$ are non-zero.} \big\}.$$Similarly, every time we consider an element of the quotients $\Q_p / p^{-n} \Z_p$ it will be chosen from the complete system of representatives   $$\{1\} \cup \big\{ \sum_{k =n+1}^\infty \lambda_k p^{-k} \, : \, \text{only finitely many $\lambda_k$ are non-zero} \big\}.$$ 
\end{rem}

By the Peter-Weyl theorem the elements of $\widehat{\Z}_p^d$ constitute an orthonormal basis for the Hilbert space $L^2 (\Z_p^d)$, which provide us a Fourier analysis for suitable functions defined on $\Z_p$ in such a way that the formula $$\varphi(u) = \sum_{\tau \in \widehat{\Z}_p^d} \widehat{\varphi}(\tau) \chi_p (\tau  u),$$holds almost everywhere in $\Z_p$. Here $\mathcal{F}_{\Z_p^d}[\varphi]$ denotes the Fourier transform of $f$, in turn defined as $$\mathcal{F}_{\Z_p^d}[\varphi](\tau):= \int_{\Z_p^d} \varphi(u) \overline{\chi_p (\tau  u)}du,$$where $du$ is the normalised Haar measure on $\Z_p^d$. 

Another important tool from $p$-adic analysis that we are going to be using, is the calculation of the $p$-adic Gaussian integral on the disk, see \cite[p.p. 65]{vladiBook}. 
\begin{lema}\label{lemagaussian}
    Let $p \neq 2$ be a prime number, and take any $\gamma \in \Z$. Then, for any $a,b \in \Q_p$, we have \[ \int_{p^\gamma \Z_p} e^{2 \pi i \{ a u^2 + b u\}_p} du =  \begin{cases}
        p^{-\gamma} \1_{p^{-\gamma} \Z_p} (b), & \, \, \text{if} \, \, \, |a|_p \leq  p^{2 \gamma}, \\ \lambda_p(a) |a|_p^{-1/2} e^{2 \pi i \{ \frac{- b^2}{4a}  \}_p }\1_{p^{\gamma}\Z_p} (b/a)& \, \, \, \text{if} \, \, \, |a|_p >  p^{2 \gamma}.
    \end{cases}\]Here $\lambda_p : \Q_p \to \C$ is the function defined as \[\lambda_p (a) := \begin{cases}
        1, & \, \, \text{if} \, \, ord(u) \, \text{is even,} \\
         \big( \frac{a_0}{p} \big), & \, \, \text{if} \, \, ord(u) \, \text{is odd and} \, \, p \equiv 1 (\mathrm{mod} \, 4), \\ 
         i\big( \frac{a_0}{p} \big), & \, \, \text{if} \, \, ord(u) \, \text{is odd and} \, \, p \equiv 3 (\mathrm{mod} \, 4),
    \end{cases}\]
where $(\frac{a_0}{p})$ denotes the Legendre symbol of $a_0$. Here, to simplify our notation, we are going to define $$\Lambda (a,b):= \lambda_p(a) |a|_p^{-1/2} e^{2 \pi i \{ \frac{- b^2}{4a}  \}_p }, $$ so that $|\Lambda (a,b)|= |a|_p^{-1/2}$, and in this way $$\int_{p^\gamma \Z_p} e^{2 \pi i \{ a u^2 + b u\}_p} du =  p^{-\gamma} \1_{p^{-2\gamma} \Z_p}(a) \cdot \1_{p^{-\gamma} \Z_p} (b)    + \1_{\Q_p \setminus p^{-2\gamma} \Z_p}(a) \cdot \Lambda (a,b) \cdot \1_{p^{\gamma}\Z_p} (b/a) .  $$
\end{lema}

With this tool we can prove the following auxiliary lemma: 
\begin{lema}\label{lemaaux}
    Let $(x_2, x_3) \in \Z_p^2$, and take $\xi_3 , \xi_4 \in \widehat{\Z}_p$. Then $$\int_{\Z_p^2} \Big| \int_{ \Z_p}e^{2 \pi i \{(\xi_3   x_2 +  \xi_4 x_3)u +  \xi_4 x_2 \frac{u^2}{2} \}_p} du \Big|^2 dx_2 dx_3 = \max \{|\xi_3|_p , |\xi_4|_p \}^{-1} = \|(\xi_3 , \xi_4) \|_p^{-1}.$$
\end{lema}
\begin{proof}
    Let us define the auxiliary function $$f(x_2 , x_3):=\int_{ \Z_p}e^{2 \pi i \{(\xi_3   x_2 +  \xi_4 x_3)u +  \xi_4 x_2 \frac{u^2}{2} \}_p} du. $$Then, as a function on $\Z_p^2$, we can compute its Fourier transform  as \begin{align*}
        \mathcal{F}_{\Z_p^2}[f](\alpha , \beta)&= \int_{\Z_p^2} \int_{ \Z_p}e^{2 \pi i \{(\xi_3   x_2 +  \xi_4 x_3)u +  \xi_4 x_2 \frac{u^2}{2} - \alpha x_2 - \beta x_3 \}_p} du dx_2 dx_3\\&=\int_{ \Z_p}\int_{\Z_p^2} e^{2 \pi i \{(\xi_3 u  +  \xi_4 \frac{u^2}{2} - \alpha )x_2 - ( \xi_4 u - \beta)x_3 \}_p}  dx_2 dx_3du\\&= \int_{\Z_p} \1_{\Z_p} (\xi_3 u  +  \xi_4 \frac{u^2}{2} - \alpha) \1_{\Z_p} (\xi_4 u - \beta) du.  
    \end{align*}The condition $|\xi_4 u - \beta|_p \leq 1 $ implies that $\xi_4 u =\beta$ in $\Q_p / \Z_p$. In this way, we have the following equality modulo $\Z_p$:  $$\xi_3 u  +  \xi_4 \frac{u^2}{2} = \frac{u}{2}\big( 2 \xi_3   +  \xi_4 u\big) = \frac{u}{2}\big( 2 \xi_3   +  \beta \big)  \in \Q_p / \Z_p,$$therefore \begin{align*}
        \mathcal{F}_{\Z_p^2}[f](\alpha , \beta)&= \int_{\Z_p} \1_{\Z_p} (u(  \xi_3   +  \beta/2) - \alpha) \1_{\Z_p} (\xi_4 u - \beta) du \\ &=\int_{\Z_p} \1_{\frac{\alpha}{\xi_3   +  \beta/2} + p^{-\vartheta(\xi_3   +  \beta/2)} \Z_p} (u) \1_{\frac{\beta}{\xi_4} + p^{-\vartheta(\xi_4)}\Z_p }(u)du \\ &= \mu_{\Z_p} \Big(\frac{\alpha}{\xi_3   +  \beta/2} + p^{-\vartheta(\xi_3   +  \beta/2)} \Z_p\cap\frac{\beta}{\xi_4} + p^{-\vartheta(\xi_4)} \Z_p \Big),
    \end{align*}where $\mu_{\Z_p}$ denotes the unique normalized Haar measure on $\Z_p$. Here we have to notice how, since we are integrating over $u\in \Z_p$, in order for the above integral to be different from zero we need $$\big| \frac{\alpha}{\xi_3   +  \beta/2}\big|_p \leq 1, \, \, \text{and} \, \, \, \big| \frac{\beta}{\xi_4}\big|_p \leq 1,$$so that only finitely many terms  $\mathcal{F}_{\Z_p^2}[f](\alpha , \beta)$ can be non-zero. Also \[ \mathcal{F}_{\Z_p^2}[f](\alpha , \beta) = \begin{cases}
    0, & \, \, \text{if the balls are disjoint}, \\ 
    \max\{|\xi_4|_p ,  |\xi_3   +  \beta/2|_p\}^{-1}, &\, \, \text{if the balls are not disjoint}.
\end{cases}\] 
Finally, we use this Fourier transform to calculate the $L^2$-norm of $f$ in two different cases. When $|\xi_3|_p >|\xi_4|_p$, we have $|\xi_3   +  \beta/2|_p = |\xi_3|_p$, so that  
\begin{align*}
    \| f \|_{L^2 (\Z_p^2)}^2 & = \sum_{(\alpha , \beta) \in \widehat{\Z}_p^2} \mu_{\Z_p} \Big(\frac{\alpha}{\xi_3   +  \beta/2} + p^{-\vartheta(\xi_3   +  \beta/2)} \Z_p\cap\frac{\beta}{\xi_4} + p^{-\vartheta(\xi_4)} \Z_p \Big)^2 \\ &=\sum_{(\alpha , \beta) \in \widehat{\Z}_p^2} \mu_{\Z_p} \Big(\frac{\alpha}{\xi_3   +  \beta/2} + p^{-\vartheta(\xi_3   )} \Z_p\cap\frac{\beta}{\xi_4} + p^{-\vartheta(\xi_4)} \Z_p \Big)^2 \\ &=\sum_{\beta \in \widehat{\Z}_p}|\xi_3|_p^{-1} \sum_{\alpha \in \widehat{\Z}_p} \mu_{\Z_p} \Big(\frac{\alpha}{\xi_3   +  \beta/2} + p^{-\vartheta(\xi_3   +  \beta/2)} \Z_p\cap\frac{\beta}{\xi_4} + p^{-\vartheta(\xi_4)} \Z_p \Big) \\ &=\sum_{\beta \in \widehat{\Z}_p}|\xi_3|^{-1}_p   \mu_{\Z_p}\Big(\frac{\beta}{\xi_4} + p^{- \vartheta(\xi_4)} \Z_p \Big) = |\xi_3|_p^{-1}.
\end{align*}
When $|\xi_3|_p \leq |\xi_4|_p,$we get $|\xi_3   +  \beta/2|_p \leq |\xi_4|_p$, so we conclude in a similar way

\begin{align*}
    \| f \|_{L^2 (\Z_p^2)}^2 & = \sum_{(\alpha , \beta) \in \widehat{\Z}_p^2} \mu_{\Z_p} \Big(\frac{\alpha}{\xi_3   +  \beta/2} + p^{-\vartheta(\xi_3   +  \beta/2)} \Z_p\cap\frac{\beta}{\xi_4} + p^{-\vartheta(\xi_4)} \Z_p \Big)^2 \\ &= \sum_{\alpha \in \widehat{\Z}_p} \sum_{\beta \in \widehat{\Z}_p} |\xi_4|_p^{-1} \mu_{\Z_p} \Big(\frac{\alpha}{\xi_3   +  \beta/2} + p^{-\vartheta(\xi_3   +  \beta/2)} \Z_p\cap\frac{\beta}{\xi_4} + p^{-\vartheta(\xi_4)} \Z_p \Big) \\ &=\sum_{\alpha \in \widehat{\Z}_p}|\xi_4|^{-1}_p  \mu_{\Z_p}\Big(\frac{\alpha}{\xi_3   +  \beta/2} + p^{-\vartheta(\xi_3 +  \beta/2)} \Z_p \Big) = |\xi_4|_p^{-1}.
\end{align*}
This concludes the proof. 
\end{proof}
\subsection{The Engel group over $\Z_p$}
Let $p>2$ be a prime number. Let us denote by $\mathfrak{b}_4$ the $\Z_p$-Lie algebra generated by $X_1,..,X_4$, with the commutation relations $$[X_1 , X_2] = X_3, \,\,\, [X_1 , X_3 ] = X_4, \,\,\, \mathfrak{b}_4 = \bigoplus_{j=1}^3 \mathcal{V}_j, $$where $$\mathcal{V}_1 = \mathrm{span}_{\Z_p} \{ X_1 , X_2 \}, \,\,\, \mathcal{V}_1 = \mathrm{span}_{\Z_p} \{ X_3 \}, \,\,\, \mathcal{V}_3 = \mathrm{span}_{\Z_p} \{ X_4 \}.$$We call the $\Z_p$-Lie algebra $\mathfrak{b}_4$ the $4$-dimensional \emph{Engel algebra}, and its exponential image, which we denote here by $\mathcal{B}_4$, is called the \emph{Engel group} over the $p$-adic integers.    
Let us consider the realization of $\mathfrak{b}_4$ as the matrix algebra  \[
\mathfrak{b}_{4}(\Z_p)= \left\{X=
  \begin{bmatrix}
    0 & x_1 & 0 & x_4 -\frac{1}{2}x_1 (x_3  - \frac{1}{2}x_1 x_2) - \frac{1}{6}x_1^2 x_2 \\
    0 & 0 & x_1 & x_3 - \frac{1}{2}x_1 x_2 \\
    0 & 0 & 0 & x_2  \\
    0 & 0 & 0 & 0 \\
  \end{bmatrix}\in \mathcal{M}_{4}(\Z_p) \, : \, (x_1, x_2, x_3, x_4) \in \Z_p^4 \right\}. 
\]
With this realization, and using the usual matrix exponential map, we can think on $\mathcal{B}(\Z_p)$ as the matrix group \[
\mathcal{B}_{4}(\Z_p)= \left\{\mathbf{x}=
  \begin{bmatrix}
    1 & x_1 & \frac{1}{2}x_1^2 & x_4  \\
    0 & 1 & x_1 & x_3  \\
    0 & 0 & 1 & x_2  \\
    0 & 0 & 0 & 1 \\
  \end{bmatrix}\in \mathcal{M}_{4}(\Z_p) \, : \, (x_1, x_2, x_3, x_4) \in \Z_p^4 \right\}, 
\]
which analytically isomorphic to $\Z_p^4$ with the operation $$\mathbf{x} \star \mathbf{y}:=(x_1 + y_1, x_2 + y_2, x_3 + y_3 + x_1 y_2, x_4 + y_4 + x_1 y_3 + \frac{1}{2} x_1^2 y_2).$$

The exponential map transforms sub-ideals of the Lie algebra $\mathfrak{b}_{4} $ to subgroups of $\mathcal{B}_4 \cong (\mathfrak{b}_4 , \star)$, which can be endowed with the sequence of subgroups $J_n := (  \mathfrak{b}_{4}(p^n\Z_p), \star)$, where $$ \mathfrak{b}_{4}(p^n\Z_p)= p^n \Z_p X_1 + p^n \Z_p X_2 + p^n \Z_p X_3 + p^n\Z_p X_{4},$$so $\mathcal{B}_4$ is a compact Vilenkin group, together with the sequence of compact open subgroups $$G_n := \mathcal{B}_4 (p^n \Z_p)= \textbf{exp}(\mathfrak{b}_{4}(p^n \Z_p)), \,\,\, n \in \N_0.$$
  Notice how the sequence $\mathscr{G}=\{G_n\}_{n \in \N_0}$ forms a basis of neighbourhoods at the identity, so the group is metrizable, and we can endow it with the natural ultrametric \[ | \mathbf{x} \star \mathbf{y}^{-1}|_{\mathscr{G}} :=\begin{cases} 0 & \, \, \text{if} \, \mathbf{x}=\mathbf{y}, \\ |G_n| = p^{-4n}  & \, \, \text{if} \, \mathbf{x} \star \mathbf{y}^{-1} \in G_n \setminus G_{n+1}.\end{cases}\]   
Nevertheless, instead of this ultrametric we will use the $p$-adic norm $$\| \mathbf{x} \|_p:= \max_{1 \leq j \leq 4} \|x_j\|_p, .$$ Notice that how $\|\mathbf{x} \|_p^{4} = |\mathbf{x}|_{\mathscr{G}},$ for any $\mathbf{x} \in \mathcal{B}_4 (\Z_p).$
\subsection{Directional VT operators}
One important idea from the theory of differential and pseudo-differential operators on Lie groups, is the correspondence between directional derivatives and elements of the Lie algebra. However, in the $p$-adic case, there are plenty of non-trivial locally constant functions, due to the fact that $p$-adic numbers are totally disconnected. This means that the usual notion of derivative does not apply, and therefore we need to find an alternative kind of operators to talk about differentiability on these groups. A first approach to this problem can be the \emph{Vladimitov-Taibleson operator} \cite{Dragovich2023, Dragovich2017}, which we define for general compact $\K$-Lie groups as follows:    

\begin{defi}\label{defiVToperator}\normalfont
Let $\K$ be a non-archimedean local field with ring of integers $\mathscr{O}_\K$, prime ideal $\mathfrak{p}=\textbf{p} \mathscr{O}_\K$ and residue field $\mathbb{F}_q \cong \mathscr{O}_\K/\textbf{p} \mathscr{O}_\K$. Let  $\mathbb{G} \leq \mathrm{GL}_m (\mathscr{O}_\K)$ be a compact $d$-dimensional $\K$-Lie group. We define the \emph{Vladimirov-Taibleson operator} on $G$  via the formula \[
D^\alpha f(\mathbf{x}) := \frac{1 - q^\alpha}{1 - q^{- (\alpha + d)}} \int_{\mathbb{G}} \frac{f (\mathbf{x} \star \mathbf{y}^{-1}) - f(\mathbf{x})}{\|\mathbf{y} \|_\K^{ \alpha + d}} d\mathbf{y},
\]where \[\| \mathbf{y} \|_\K := \begin{cases}
    1, \, & \, \, \text{if} \, \, \mathbf{y} \in \mathbb{G} \setminus \mathrm{GL}_m (\textbf{p}\mathscr{O}_\K), \\ q^{-n}, \, & \, \, \text{if} \, \, y \in \mathrm{GL}_m (\textbf{p}^n\mathscr{O}_\K) \setminus \mathrm{GL}_m (\textbf{p}^{n+1}\mathscr{O}_\K). 
\end{cases}\]Here $dy$ denotes the unique normalized Haar measure on $\mathbb{G}$. Sometimes it will be convenient to consider the operator \[
\mathbb{D}^\alpha f(\mathbf{x}) :=\frac{1-q^{-d}}{1-q^{-(\alpha +d)}}f(\mathbf{x}) + \frac{1 - q^\alpha}{1 - q^{- (\alpha + d)}} \int_{\mathbb{G}} \frac{f (\mathbf{x} \star \mathbf{y}^{-1}) - f(\mathbf{x})}{\|\mathbf{y} \|_\K^{ \alpha + d}} d \mathbf{y},
\]
\end{defi}

The Vladimirov-Taibleson operator can be considered as a fractional Laplacian for functions on totally disconnected spaces, and it provides a first notion of differentiability. However, for functions of several variables it is natural to consider the differentiability of the function in each variable, or in a certain given direction. For that reason, we introduce the following definition:

 \begin{defi}\label{defcompactdirectionalVT}\normalfont
      Let $\mathfrak{g}$ be the $\mathscr{O}_\K$-Lie module associated to $\mathbb{G}$, and assume that $\mathfrak{g}$ is nilpotent. Let $\alpha>0$. Given a $V \in \mathfrak{g}$, we define the \emph{directional VT operator of order $\alpha$ in the direction of} $V$ as the linear invariant operator $\partial_V^\alpha$ acting on smooth functions via the formula $$\partial_V^\alpha f (\mathbf{x}) := \frac{1 - q^{\alpha}}{1-q^{-(\alpha+1)}} \int_{\mathscr{O}_\K} \frac{f(\mathbf{x} \star  \textbf{exp}(tV)^{-1}) - f(\mathbf{x})}{|t|_\K^{ \alpha + 1}}dt, \, \, \, f \in \mathcal{D}(\mathbb{G}).$$ 
 \end{defi} 

 \begin{rem}
 Directional VT operators are interesting because they associate a certain pseudo-differential operator to each element of the Lie algebra. Nonetheless, it is important to remark how this association does not follow the same patter as in the locally connected case, where the correspondence between vectors and operators preserves the Lie algebra structure.
 \end{rem}

Just to give an example, in the particular case when $G=\mathfrak{g}=\Z_p^d$, these operators take the form $$\partial_V^\alpha f (x) := \frac{1 - p^{\alpha}}{1-p^{-(\alpha+1)}}\int_{\Z_p^d} \frac{f(x -tV) - f(x)}{| t |_p^{ \alpha + 1}}dt, $$and one can easily compute its associated symbol: $\sigma_{\partial_V^\alpha}(\xi)=0$, if $V \cdot \xi =0$, and $$\sigma_{\partial_V^\alpha}(\xi) = |V \cdot \xi|_p^\alpha - \frac{1 - p^{-1}}{1 - p^{- (\alpha + 1)}}, \, \, \, \, \xi \in \widehat{\Z}_p^d,$$in other case. If we define $\partial_{x_i}^\alpha := \partial_{e_i}^\alpha $, where $e_i$, $1 \leq i \leq d$ are the canonical vectors of $\Q_p^d$, then \[\sigma_{\partial_{x_i}^\alpha}(\xi) = \begin{cases}
    0, \, & \, \, \text{if} \, \, |\xi_i|_p=1,\\| \xi_i|_p^\alpha - \frac{1 - p^{-1}}{1 - p^{- (\alpha + 1)}}  & \, \, \text{if} \, \, |\xi_i|_p>1,
\end{cases}
 \]which resembles the symbol of the usual partial derivatives on $\R^d$, justifying that way our choice of notation. However, we want to be emphatic about the fact that these directional VT operators are not derivatives, but rather some special kind of pseudo-differential operators which we will study with the help of the Fourier analysis on compact groups.

Before proceeding to the next section, let us introduce some notation. 

\begin{defi}\label{definotation}\normalfont
\,
\begin{itemize}
        \item The symbol $\mathrm{Rep}(\mathcal{B}_4)$ will denote the collection of all unitary finite-dimensional representations of $\mathcal{B}_4$. We will denote by $\widehat{\mathcal{B}}_4$ the \emph{unitary dual of $\mathcal{B}_4$}.
        \item Let $K$ be a normal sub-group of $\mathcal{B}_4$. We denote by $K^\bot$ the \emph{anihilator of $K$}, here defined as $$K^\bot:= \{ [\pi] \in \mathrm{Rep}(\mathcal{B}_4) \, : \, \pi|_{K}=I_{d_\pi} \}.$$Also, we will use the notation $$B_{\widehat{\mathcal{B}}_4}(n):=\widehat{\mathcal{B}}_4 \cap \mathcal{B}_4(p^n\Z_p)^\bot,$$and $\widehat{\mathcal{B}}_4(n):= B_{\widehat{\mathcal{B}}_4}(n) \setminus B_{\widehat{\mathcal{B}}_4}(n-1).$ 
        \item We say that a function $f:\mathcal{B}_4 \to \C$ is a \emph{smooth function}, if $f$ is a locally constant function with a fixed index of local constancy, i.e., there is an $n_f \in \N_0$, which we always choose to be the minimum possible, such that $$f(\mathbf{x}\star \mathbf{y}) = f(\mathbf{x}), \, \, \,\text{ for all} \,\,\, y \in \mathcal{B}_4(p^n \Z_p).$$
        We will denote by $\mathcal{D}(\mathcal{B}_4)$ the collection of all smooth functions on $\mathcal{B}_4$, and $\mathcal{D}_n (\mathcal{B}_4)$ will denote the collection of smooth functions with index of local constancy equal to $n \in \N_0$. 
    \end{itemize}
\end{defi}

\section{Representation theory of the Engel group}
Given a representation $[\pi] \in \widehat{\mathcal{B}}_4$ we can say that there are two possibilities: the representation is trivial on the center, or it is not. When the representation is trivial on the center, since $\mathcal{B}_4 / \mathcal{Z}(\mathcal{B}_4) \cong \mathbb{H}_1 (\Z_p)$, the representation must descend to an unitary irreducible representation of $\mathbb{H}_1 (\Z_p)$. These representations were completely described in \cite{SublaplacianHeisenberg}, and they have the form  

$$\pi_{(\xi_1 ,\xi_2, \xi_3)}(\mathbf{x}) \varphi (u) := e^{2 \pi i \{\xi_1 x_1 + \xi_2 x_2 + \xi_3 (x_3 +  u x_2) \}_p} \varphi (u + x_1), \, \, \, \, \varphi \in \mathcal{H}_{\xi_3},$$
where
$$\mathcal{H}_{\xi_3} := \mathrm{span}_\C \{ \varphi_h \, : \, h \in \Z_p / p^{-\vartheta(\xi_3)} \Z_p  \}, \, \, \, \varphi_h (u) := |\xi_3|^{1/2}_p \1_{h + p^{-\vartheta(\xi_3)} \Z_p} (u), \, \, \mathrm{dim}_\C(\mathcal{H}_{\xi_3}) =|\xi_3|_p.$$
See \cite{SublaplacianHeisenberg} for the full calculation of $\widehat{\mathbb{H}}_d(\Z_p)$. This shorten our calculations here by a lot, and actually, the ideas introduced in \cite{SublaplacianHeisenberg} are very useful to deduce how to obtain the unitary dual for the more complicated case of $\mathcal{B}_4$. 

\subsection{Level 1:}
We will use the fact that any representation of $\mathcal{B}_4$ reduces to the representation of some finite group of Lie type, obtained as the quotient between $\mathcal{B}_4$ and one of it compact open subgroups. See again the notation introduced in Definition \ref{definotation}. 

First, as it is obvious, $B_{\widehat{\mathcal{B}}_4} (0)$, the collection of unitary irreducible representations that are trivial on $G_0 = \mathcal{B}_4 (\Z_p)$, contains only the identity representation. To calculate $B_{\widehat{\mathcal{B}}_4} (1)$, the collection of unitary irreducible representations that are trivial on $G_1 = \mathcal{B}_4 (p \Z_p)$, we have to find those representations which descend to an element of the unitary dual of $G / G_1 \cong \mathcal{B}_4 (\Z_p / p \Z_p) \cong \mathcal{B}_4 (\mathbb{F}_p) .$ The usual representation theory of nilpotent groups applied to $\mathcal{B}_4 (\mathbb{F}_p)$ provides us with three different kind of representations, which correspond to the three different kind of co-adjoint orbits obtained via the Kirilov orbit method:

\begin{itemize}
    \item We have $p^2$ one dimensional representations given by the $p$-adic characters $$\pi_{(\xi_1 , \xi_2 ,1,1)} (\mathbf{x}) = \chi_{\xi_1 , \xi_2} (\mathbf{x}) = e^{2 \pi i \{ \xi_1 x_1 + \xi_2 x_2 \}_p}, \, \, \, \, \| (\xi_1 , \xi_2) \|_p \leq p.$$Clearly it holds $\int_{\mathcal{B}_4} |\chi_{\xi_1 , \xi_2} (\mathbf{x})|^2 d\mathbf{x} =1$.
    \item There are $p-1$ noncommutative representations which are trivial on the center of $\mathcal{B}_4 (\mathbb{F}_p)$, and therefore descend to a representation of $\mathbb{H}_1 (\mathbb{F}_p)$. These have the form $$\pi_{(1,1 , \xi_3 ,1)} (\mathbf{x}) = \pi_{ \xi_3}(\mathbf{x}) \varphi (u) := e^{2 \pi i \{ \xi_3 (x_3 +  u x_2) \}_p} \varphi (u + x_1), \, \, \, \, \varphi \in \mathcal{H}_{\xi_3},$$where $|\xi_3|_p = p$, and  $$\mathcal{H}_{\xi_3} := \spn_\C \{ \varphi_h \, : \, h \in \Z_p / p \Z_p  \}, \, \, \, \varphi_h (u) := p^{ 1/2} \1_{h + p \Z_p} (u), \, \, \dim_\C(\mathcal{H}_{\xi_3}) =p.$$The characters associated to these representations have the form $$\chi_{\pi_{\xi_3}}(\mathbf{x}) = \sum_{h \in  \Z_p/p\Z_p} e^{ 2 \pi i \{ {\xi_3}(x_3 +  h x_2 ) \}_p} \1_{p \Z_p}(x_1) = p e^{ 2 \pi i \{ {\xi_3} x_3  \}_p} \1_{p \Z_p}(x_1) \1_{p \Z_p}(x_2), $$thus the irreductibility of $[\pi_{\xi_3}]$ is proved by the condition  $$\int_{\mathcal{B}_4}|\chi_{\pi_{\xi_3}}(\mathbf{x})|^2 d\mathbf{x} \,  = p^{2} \int_{\mathcal{B}_4}|\1_{p \Z_p}(x_1) \1_{p \Z_p}(x_2)|^2 d\mathbf{x} \,  = 1.$$
    \item There are $p(p-1)$ representations with the form
 $$\pi_{(1,\xi_2,1, \xi_4,1)} (\mathbf{x}) \varphi(u) = \pi_{(\xi_2 , \xi_4)} (\mathbf{x}) \varphi(u):= e^{2 \pi i \{ \xi_4(x_4 +ux_3 + \frac{1}{2} u^2 x_2 ) + \xi_2 x_2 \}_p} \varphi (u + x_1 ),  \, \varphi \in \mathcal{H}_{\xi_4},$$where $1 \leq |\xi_2|_p \leq p$ and $ |\xi_4|_p = p$. The representation space needs to be  $$\mathcal{H}_{\xi_4} := \spn_\C \{ \varphi_h \, : \, h \in \Z_p / p \Z_p  \}, \, \, \, \varphi_h (u) := p^{ 1/2} \1_{h + p \Z_p} (u), \, \, \dim_\C(\mathcal{H}_{\xi_4}) =p.$$By using this basis we can compute \begin{align*}
     \chi_{\pi_{(\xi_2 , \xi_4)}}(\mathbf{x}) = pe^{2 \pi i \{ \xi_2 x_2 + \xi_4 x_4 \}_p} \1_{p\Z_p}(x_1) \int_{\Z_p} e^{2 \pi i \{ \xi_4 x_3 u + \frac{1}{2}\xi_4 x_2 u^2 \}_p} du,
 \end{align*} 
where the above integral is a $p$-adic Gaussian integral. Using Lemma \ref{lemagaussian}, see the calculations in  \cite[{p.p. 65}]{vladiBook}, we obtain  
 $\chi_{\pi_{(\xi_4 , \xi_2)}}(\mathbf{x})=$\[ \, \, \, \, \, \,\,\,\,\,  pe^{2 \pi i \{ \xi_2 x_2 + \xi_4 x_4 \}_p} \1_{p\Z_p}(x_1) \times \begin{cases}
    \1_{\Z_p}(\xi_4 x_3), \, & \, \, \text{if} \, \, |\xi_4 x_2|_p \leq 1,\\ \lambda_p (\frac{\xi_4 x_2}{2})| \xi_4 x_2|_p^{-1/2} e^{2 \pi i \{ -\frac{(\xi_4 x_3)^2}{2 \xi_4 x_2} \}_p}\1_{\Z_p} (\frac{x_3}{ x_2}),  & \, \, \text{if} \, \, |\xi_4 x_2|_p>1.
\end{cases}
 \]See again Lemma \ref{lemagaussian}. Now, considering the fact that $|\xi_4|_p=p$, the irreducibility of the representation follows from the calculation \begin{align*}
     \int_{\mathcal{B}_4} &|\chi_{\pi_{(\xi_4 , \xi_2)}}(\mathbf{x})|^2 dx = \int_{\Z_p^3}\int_{p\Z_p} p^2 \1_{p\Z_p}(x_1)\1_{p\Z_p}(x_3)    dx_2 (dx_1 dx_3 dx_4) \\ &+ \int_{\Z_p^3}\int_{\Z_p \setminus \Z_p} p^2 | \1_{p\Z_p} (x_1) \lambda_p (\frac{\xi_4 x_2}{2})| \xi_4 x_2|_p^{-1/2} \1_{\Z_p} (x_3/ x_2)|^2 dx_2 (dx_1 dx_3 dx_4).
 \end{align*}
 In one hand $$\int_{\Z_p^3}\int_{p\Z_p} p^2 \1_{p\Z_p}(x_1)\1_{p\Z_p}(x_3)    dx_2 (dx_1 dx_3 dx_4) = 1/p.$$
 On the other \begin{align*}
     \int_{\Z_p^3}\int_{\Z_p \setminus \Z_p} p^2 | \1_{p\Z_p} (x_1) \lambda_p (\frac{\xi_4 x_2}{2})| \xi_4 x_2|_p^{-1/2} \1_{\Z_p} (x_3/ x_2)|^2 dx_2 (dx_1 dx_3 dx_4) = \\ \int_{\Z_p^3}\int_{\Z_p \setminus p \Z_p} p^2 | \1_{p\Z_p} (x_1) p^{-1/2} |x_2|_p^{-1/2} \1_{p\Z_p} (x_3)|^2 dx_2 (dx_1 dx_3 dx_4) \\ = \int_{\Z_p^3}\int_{\Z_p \setminus \Z_p} p^2 | \1_{p\Z_p} (x_1)  \1_{p\Z_p} (x_3)|^2 dx_2 (dx_1 dx_3 dx_4) = 1 - 1/p.
\end{align*}
\end{itemize}
Notice how these representations are indexed by the sets 
\begin{align*}
    &B_{\widehat{\mathcal{B}}_4}(1)=  \\&\{  \| \xi \| \leq p  : \, 1 \leq |\xi_4|_p <|\xi_3|_p \, \wedge \,  (\xi_1 , \xi_2 , \xi_3) \in \widehat{\mathbb{H}}_1 , \, \text{or}\, \, |\xi_3|_p=1, \, \,  p = |\xi_4|_p \,  \wedge  \,  \xi_1 \in \Q_p / p^{-1} \Z_p  \},
\end{align*}
and we can check how these are all the elements in $B_{\widehat{\mathcal{B}}_4}(1)=\widehat{\mathcal{B}}_4 \cap \mathcal{B}_4(p^1 \Z_p)^\bot$ since 
     \begin{align*}
         \sum_{\| (\xi_1 ,\xi_2) \|_p \leq p}& \dim_\C(\chi_{\xi_1 , \xi_2})^2 + \sum_{| \xi_3 |_p = p } \dim_\C(\pi_{\xi_3})^2 +\sum_{|\xi_2|_p \leq | \xi_4 |_p = p } \dim_\C(\pi_{(\xi_4, \xi_2)})^2 \\&= \sum_{\| (\xi_1 , \xi_2) \|_p \leq p} 1^2 + \sum_{|\xi_3 |_p = p } p^2 +\sum_{|\xi_2|_p \leq |\xi_4 |_p = p } p^2 =p^4= |G/G_1|.
     \end{align*}

Now, for the next level, we will have to start mixing these three different kind of representations. 

\subsection{Level 2:}
We want to describe the elements of $B_{\widehat{\mathcal{B}}_4}(2)=\widehat{\mathcal{B}}_4 \cap \mathcal{B}_4 (p^2\Z_p)^\bot$. Using once again how $\mathcal{B}_4 / \mathcal{Z}(\mathcal{B}_4) \cong \mathbb{H}_1 (\Z_p)$, and the description of the unitary dual given of $\mathbb{H}_d(\Z_p)$ given in \cite{SublaplacianHeisenberg}, we will obtain the new representations in $\widehat{\mathcal{B}}_4 (2)$ by considering three cases. 

The first case is when an element of $B_{\widehat{\mathcal{B}}_4}(2)$ acts trivially on the center, and it must descend to a representation of $\mathbb{H}_1(\Z_p)$, so it has the form $$\pi_\xi (\mathbf{x}) \varphi(u) = \pi_{(\xi_1 ,\xi_2, \xi_3, 1)}(\mathbf{x}) \varphi (u) := e^{2 \pi i \{\xi_1 x_1 + \xi_2 x_2 + \xi_3 (x_3 +  u x_2) \}_p} \varphi (u + x_1), \, \, \, \, \varphi \in \mathcal{H}_{\xi_3},$$where the representation space is $$\mathcal{H}_\xi = \mathcal{H}_{\xi_3} := \mathrm{span}_\C \{ \varphi_h \, : \, h \in \Z_p / p^{- \vartheta(\xi_3)} \Z_p  \}, \, \, \, \varphi_h (u) := |\xi_3|_p^{1/2} \1_{h + p^{- \vartheta(\xi_3)} \Z_p} (u), \, \, \| \xi \|_p \leq p^2,$$ the associated character is $$\chi_{\pi_{(\xi_1, \xi_2, \xi_3,1)}} (\mathbf{x}) = |\xi_3|_p e^{2 \pi i \{\xi_1 x_1 + \xi_2 x_2 + \xi_3 x_3 \}_p} \1_{p^{- \vartheta(\xi_3)}\Z_p}(x_1)\1_{p^{- \vartheta(\xi_3)}\Z_p}(x_2),$$and we clearly have $\mathrm{dim}_\C(\mathcal{H}_{\xi_3}) =|\xi_3|_p$. For these representations it holds that $$\sum_{\xi} d_\xi^2 = p^6= |\mathbb{H}_1 (\Z_p)/\mathbb{H}_1 (p^2 \Z_p)|.$$On the other hand, if a representation is in $B_{\widehat{\mathcal{B}}_4}(2)$, and it is not trivial on the center of $\mathcal{B}_4(\Z_p)$, then we need to examine two remaining cases. 

The second case we need to examine is when $|\xi_4|_p = p^2$ and $1 \leq |\xi_2|_p \leq p^2$. Then we can define $$ \pi_{(\xi_2 , \xi_4)} (\mathbf{x}) := e^{2 \pi i \{ \xi_4(x_4 +ux_3 + \frac{1}{2} u^2 x_2 ) + \xi_2 x_2 \}_p} \varphi (u + x_1 ),  \, \varphi \in \mathcal{H}_{\xi_4},$$where $1 \leq |\xi_2|_p \leq p^2$, and$$\mathcal{H}_{\xi_4} := \spn_\C \{ \varphi_h  : \, h \in \Z_p / p^{- \vartheta(\xi_4)} \Z_p  \}, \, \, \varphi_h (u) := |\xi_4|_p^{1/2} \1_{h + p^{- \vartheta(\xi_4)}\Z_p} (u), \,  \dim_\C(\mathcal{H}_{\xi_4}) =|\xi_4|_p.$$These representations are unitary and irreducible, since their associated character is given by \begin{align*}
    &\chi_{\pi_{(\xi_4 , \xi_2)}}(\mathbf{x}) = |\xi_4|_pe^{2 \pi i \{ \xi_2 x_2 + \xi_4 x_4 \}_p} \1_{|\xi_4|_p\Z_p}(x_1)\1_{|\xi_4|_p \Z_p}( x_3)\1_{|\xi_4|_p \Z_p}( x_2) \\ &+ |\xi_4|_pe^{2 \pi i \{ \xi_2 x_2 + \xi_4 x_4 \}_p} \1_{|\xi_4|_p\Z_p}(x_1)\cdot \Lambda(\frac{\xi_4 x_2}{2}, \xi_4 x_3) \cdot \1_{\Z_p} (\frac{x_3}{ x_2}) \1_{\Q_p \setminus \Z_p}(\xi_4 x_2),  
    \end{align*}which defines an $L^2$-normalized function. Here $\Lambda (a,b)$ is defined as in Lemma \ref{lemagaussian}. See \cite[{p.p. 65}]{vladiBook} for the calculations with $p$-adic Gaussian integrals. 
    
    Finally, for the third and last case, we have to examine mixed representations, which occur when $|\xi_4|_p = p$. These have the form  \[\pi_{\xi}(\mathbf{x}) \varphi (u) :=  e^{2 \pi i \{ \xi \cdot \mathbf{x} + u(\xi_3 x_2 + \xi_4 x_3) + \frac{u^2}{2} \xi_4 x_2 \}_p}\varphi(u + x_1),\]where $1 \leq |\xi_3|_p \leq 2$, and they can be realized in the sub-spaces $\mathcal{H}_\xi$ of $L^2 (\Z_p)$ defined as $$\mathcal{H}_\xi  := \spn_\C \{ \varphi_h = \|(\xi_3, \xi_4)\|_p^{1/2} \1_{h + p^{-(\xi_3, \xi_4)} \Z_p} (u)  :  h \in \Z_p / p^{-(\xi_3, \xi_4)} \Z_p  \}.$$   
    Let us define here $d_\xi:= \max \{ |\xi_3|_p , |\xi_4|_p\}$.     With this realization and the natural choice of basis, the associated matrix coefficients of the representations are given by \begin{align*}
    (\pi_{\xi } (\mathbf{x}))_{h h'}&= (\pi_{\xi} (\mathbf{x}) \varphi_h , \varphi_{h'})_{L^2 (\Z_p)} \\ &= d_\xi \int_{\Z_p} e^{2 \pi i \{ \xi \cdot x + u(\xi_3 x_2 + \xi_4 x_3) + \frac{u^2}{2} \xi_4 x_2 \}_p} \1_h (u +x_1) \1_{h'} (u) du \\ &=d_\xi  e^{2 \pi i \{\xi \cdot x \}_p} \1_{h - h'} (x_1) \int_{h' + d_\xi \Z_p}e^{2 \pi i \{(\xi_3   x_2 +  \xi_4 x_3)u +  \xi_4 x_2 \frac{u^2}{2} \}_p}  du \\ &= e^{2 \pi i \{\xi \cdot x \}_p} e^{2 \pi i \{(\xi_3   x_2 +  \xi_4 x_3)(h') +  \xi_4 x_2 \frac{(h')^2}{2} \}_p} \1_{h - h' + d_\xi \Z_p}(x_1) \\ & \, \, \, \, \,\,\,\,  \times\Big( d_\xi \int_{d_\xi \Z_p^d} e^{2 \pi i \{(\xi_3   x_2 +  \xi_4 x_3)u +  \xi_4 x_2 \frac{u^2  + 2u h' }{2} \}_p} \Big). 
\end{align*}
From here it is clear how \begin{align*}
    \chi_{\pi_\xi } (\mathbf{x}) &= \sum_{h \in \Z_p / d_\xi \Z_p} (\pi_{\xi } (\mathbf{x}))_{h h}\\ &= \sum_{h \in \Z_p / d_\xi \Z_p} d_\xi  e^{2 \pi i \{ \xi \cdot x + u(\xi_3 x_2 + \xi_4 x_3) + \frac{u^2}{2} \xi_4 x_2 \}_p} \1_{h - h'} (x_1) \int_{h' + d_\xi \Z_p}e^{2 \pi i \{(\xi_3   x_2 +  \xi_4 x_3)u +  \xi_4 x_2 \frac{u^2}{2} \}_p}  du \\ & = d_\xi  e^{2 \pi i \{\xi \cdot x \}_p} \1_{d_\xi \Z_p} (x_1) \int_{ \Z_p}e^{2 \pi i \{(\xi_3   x_2 +  \xi_4 x_3)u +  \xi_4 x_2 \frac{u^2}{2} \}_p}  du,
\end{align*}where one can write explicitly \begin{align*}
    \int_{  \Z_p}&e^{2 \pi i \{(\xi_3   x_2 +  \xi_4 x_3)u +  \xi_4 x_2 \frac{u^2}{2} \}_p} =   \1_{\Z_p}(\frac{\xi_4 x_2}{2}) \cdot \1_{ \Z_p} (\xi_3   x_2 +  \xi_4 x_3)    \\ &+ \1_{\Q_p \setminus \Z_p}(\frac{\xi_4 x_2}{2}) \cdot \Lambda (\frac{\xi_4 x_2}{2},\xi_3   x_2 +  \xi_4 x_3) \cdot \1_{\Z_p} ((\xi_3   x_2 +  \xi_4 x_3)/(\frac{\xi_4 x_2}{2})) .
\end{align*}

Let us analyze these characters: 

\textbf{Case $|\xi_3|_p \leq |\xi_4|_p$:} In this case $d_\xi = |\xi_4|_p$. Since $p>2$, we have the equality $\1_{p^{- \vartheta(\xi_4)} \Z_p}(x_2) = \1_{\Z_p}(\frac{\xi_4 x_2}{2})$. Next \begin{align*}
        \1_{\Z_p}(\frac{\xi_4 x_2}{2}) \cdot \1_{ \Z_p} (\xi_3   x_2 +  \xi_4 x_3) &=  \1_{p^{- \vartheta(\xi_4)} \Z_p}(x_2) \cdot \1_{ \Z_p} (\xi_3   x_2 +  \xi_4 x_3)  \\ &=\1_{p^{- \vartheta(\xi_4)} \Z_p}(x_2) \cdot \1_{ \Z_p} (\xi_4 x_3) \\&= \1_{p^{- \vartheta(\xi_4)}\Z_p}(x_2) \cdot \1_{ p^{- \vartheta(\xi_4)} \Z_p} (x_3),
    \end{align*}where the equality before the last one holds because $\xi_3 x_2 \in \Z_p$ when $|\xi_3|_p \leq |\xi_4|_p$. For the remaining part of the function, we have $$(\xi_3   x_2 +  \xi_4 x_3)/(\frac{\xi_4 x_2}{2}) = 2 \big( \frac{\xi_3}{\xi_4} + \frac{x_3}{x_2} \big),$$so that actually $$\1_{\Z_p} \big( 2 \big( \frac{\xi_3}{\xi_4} + \frac{x_3}{x_2} \big) \big) = \1_{\Z_p} \big(  \frac{\xi_3}{\xi_4} + \frac{x_3}{x_2}  \big) = \1_{\Z_p} \big(  \frac{x_3}{x_2} \big) .$$
    By using all these we can rewrite \begin{align*}
    \int_{  \Z_p}&e^{2 \pi i \{(\xi_3   x_2 +  \xi_4 x_3)u +  \xi_4 x_2 \frac{u^2}{2} \}_p} =   \1_{p^{- \vartheta(\xi_4)} \Z_p}(x_2) \cdot \1_{ p^{- \vartheta(\xi_4)}\Z_p} (x_3)    \\ &+ \1_{\Q_p \setminus \Z_p}(\frac{\xi_4 x_2}{2}) \cdot e^{-2 \pi i\{\frac{\xi_3^2}{2 \xi_4}x_2 +  \xi_3 x_3 \}_p}\Lambda(\frac{\xi_4 x_2}{2},\xi_4 x_3) \cdot \1_{\Z_p} \big(  \frac{x_3}{x_2} \big),
\end{align*}and therefore 
\begin{align*}
    &\chi_{\pi_\xi} (\mathbf{x}) = \\&|\xi_4|_p e^{2 \pi i \{\xi_1 x_1 + \xi_3 x_3  + \xi_2 x_2 + \xi_4 x_4 \}_p} \1_{d_\xi \Z_p} (x_1)
  \1_{p^{- \vartheta(\xi_4)} \Z_p}(x_2) \cdot \1_{ p^{- \vartheta(\xi_4)}\Z_p} (x_3)    \\ &+ |\xi_4|_p e^{2 \pi i \{\xi_1 x_1 + \xi_3 x_3  + \xi_2 x_2 + \xi_4 x_4 \}_p} \1_{d_\xi \Z_p} (x_1)
 \1_{\Q_p \setminus \Z_p}(\frac{\xi_4 x_2}{2}) \cdot e^{-2 \pi i\{\frac{\xi_3^2}{2 \xi_4}x_2 +  \xi_3 x_3 \}_p}\Lambda(\frac{\xi_4 x_2}{2},\xi_4 x_3) \cdot \1_{\Z_p} \big(  \frac{x_3}{x_2} \big),\\&=|\xi_4|_p e^{2 \pi i \{\xi_1 x_1 + \xi_2 x_2 + \xi_4 x_4 \}_p} \1_{d_\xi \Z_p} (x_1)
  \1_{p^{- \vartheta(\xi_4)} \Z_p}(x_2) \cdot \1_{ p^{- \vartheta(\xi_4)}\Z_p} (x_3)    \\ &+ |\xi_4|_p e^{2 \pi i \{\xi_1 x_1   + (\xi_2 - \frac{\xi_3^2}{2 \xi_4}) x_2 + \xi_4 x_4 \}_p} \1_{d_\xi \Z_p} (x_1)
 \1_{\Q_p \setminus \Z_p}(\frac{\xi_4 x_2}{2}) \cdot \Lambda(\frac{\xi_4 x_2}{2},\xi_4 x_3) \cdot \1_{\Z_p} \big(  \frac{x_3}{x_2} \big). 
\end{align*}
So, in the end the variable $\xi_3$ disappears,  and the characters are given by 
\begin{align*}
    \chi_{\pi_\xi} (\mathbf{x}) &= |\xi_4|_p e^{2 \pi i \{\xi_1 x_1 + \xi_2 x_2 + \xi_4 x_4 \}_p} \1_{d_\xi \Z_p} (x_1)
  \1_{p^{- \vartheta(\xi_4)} \Z_p}(x_2) \cdot \1_{ p^{- \vartheta(\xi_4)}\Z_p} (x_3)    \\ &+ |\xi_4|_p e^{2 \pi i \{\xi_1 x_1   + \xi_2  x_2 + \xi_4 x_4 \}_p} \1_{d_\xi \Z_p} (x_1)
 \1_{\Q_p \setminus \Z_p}(\frac{\xi_4 x_2}{2}) \cdot \Lambda(\frac{\xi_4 x_2}{2},\xi_4 x_3) \cdot \1_{\Z_p} \big(  \frac{x_3}{x_2} \big)\\&=|\xi_4|_p e^{2 \pi i \{\xi_1 x_1   + \xi_2  x_2 + \xi_4 x_4 \}_p} \1_{p^{-\vartheta(\xi_4)} \Z_p} (x_1) \int_{\Z_p}e^{2 \pi i \{ \xi_4 x_3u +  \xi_4 x_2 \frac{u^2}{2} \}_p}du.
\end{align*}
Counting all the different functions among these, we can see how we have  $(p-1)p^3$ different representations, all with dimension $d_\xi = |\xi_4|_p =p^2$.

\textbf{Case $p^2 = |\xi_3|_p > |\xi_4|_p$:} In this case $d_\xi = |\xi_3|_p$ and $|\xi_1|_p = |\xi_2|_p =1$, so that $\xi_1$ and $\xi_2$ are not appearing. The associated character is \begin{align*}
    \chi_{\pi_\xi } (\mathbf{x}) &= |\xi_3|_p  e^{2 \pi i \{ \xi_3 x_3   + \xi_4 x_4 \}_p} \1_{p^{-\vartheta(\xi_3)} \Z_p} (x_1) \int_{ \Z_p}e^{2 \pi i \{(\xi_3   x_2 +  \xi_4 x_3)u +  \xi_4 x_2 \frac{u^2}{2} \}_p}  du, \\ &= |\xi_3|_p  e^{2 \pi i \{ \xi_3 x_3  +  \xi_4 x_4 \}_p} \1_{p^{-\vartheta(\xi_3)} \Z_p} (x_1) \int_{ \Z_p}e^{2 \pi i \{(\xi_3   x_2 +  \xi_4 x_3)u +  \xi_4 x_2 \frac{u^2}{2} \}_p} du.
\end{align*}Counting all the different functions among these, we can see how we have  $(p-1)^2 p$ different representations, all with dimension $d_\xi = |\xi_3|_p$. 

Summing up, all the obtained representations can be indexed by the set \begin{align*}
    &B_{\widehat{\mathcal{B}}_4}(2)\\&= \{  \| \xi \| \leq p^2  : \, 1 \leq |\xi_4|_p <|\xi_3|_p, \,  \,  (\xi_1 , \xi_2 , \xi_3) \in \widehat{\mathbb{H}}_1 , \, \text{or} \,  |\xi_3|_p = 1, \,  \,  \xi_1 \in \Q_p / p^{\vartheta(\xi_4)} \Z_p, \, \,\xi_2 \in \widehat{\Z}_p \}
\end{align*}
To convince ourselves that these representations are indeed irreducible we just apply Lemma \ref{lemaaux} to see that \begin{align*}
    \int_{\mathcal{B}_4}|\chi_{\pi_\xi } (\mathbf{x})|^2 d\mathbf{x} &= \| (\xi_3, \xi_4)\|_p^2 \int_{\Z_p^4}  \Big| \1_{p^{-\vartheta(\xi_3, \xi_4)} \Z_p} (x_1) \int_{ \Z_p}e^{2 \pi i \{(\xi_3   x_2 +  \xi_4 x_3)u +  \xi_4 x_2 \frac{u^2}{2} \}_p}  du\Big|^2 d\mathbf{x} = 1. 
\end{align*}Notice how these are indeed all the desired representations since  \begin{align*}
    \sum_{\xi \in B_{\widehat{\mathcal{B}}_4}(2)} d_\xi^2 &=\sum_{\| \xi \|_p\leq p^2, \, (\xi_1 , \xi_2, \xi_3) \in  \widehat{\mathbb{H}}_1 } d_\xi^2 +  \sum_{ |\xi_4|_p = p \, \wedge |\xi_3|_p=1} d_\xi^2 +\sum_{ |\xi_4|_p = p \, \wedge |\xi_3|_p=p^2} d_\xi^2 + \sum_{ |\xi_4|_p = p^2} d_\xi^2\\ &= p^6 + \sum_{ |\xi_4|_p = p \, \wedge |\xi_3|_p=1} p^2 + \sum_{ |\xi_4|_p = p \, \wedge |\xi_3|_p=p^2} (p^2)^2 + \sum_{ |\xi_4|_p = p^2} (p^2)^2 \\ &=p^6 + (p-1)p^3(p)^2 +(p-1)^2p(p^2)^2 +p^3 (p-1)(p^2)^2 \\ &=p^6 +(p-1)p^5 + (p-1)^2p^7 +(p-1)p^7 \\ &= p^6 +(p-1)p^5 + (p-1)^2p^5 + (p-1)p^7\\ &=p^8=  (p^2)^4 = |\mathcal{B}_4 (\mathbb{F}_{p^2})|= |G/G_2|.
\end{align*}

\subsection{General case and Fourier series}
We can use all the ideas collected so far to deduce the general case. First, we can see how the representations we are looking for must have the form 
\[\pi_{\xi}(\mathbf{x}) \varphi (u) := 
    e^{2 \pi i \{\xi_1 x_1 + \xi_2 x_2 + \xi_3 (x_3 +  u x_2)  + \xi_4 (x_4 +  u x_3 + \frac{u^2}{2} x_2) \}_p} \varphi (u + x_1) =e^{2 \pi i \{ \xi \cdot \mathbf{x} + u(\xi_3 x_2 + \xi_4 x_3) + \frac{u^2}{2} \xi_4 x_2 \}_p}\varphi(u + x_1) , \]
where we are taking the function $\varphi$ to be an element of the following sub-space of $L^2 (\Z_p):$ 
\[ \mathcal{H}_\xi :=  \mathrm{span}_\C \{ \| (\xi_3, \xi_4)\|_p^{1/2} \1_{h + d_\xi\Z_p } \,\, : \, h \in \Z_p / p^{-\vartheta(\xi_3, \xi_4)} \Z_p \}, \, \, \, d_\xi:= \mathrm{dim}_\C (\mathcal{H}_\xi) =  \|(\xi_3,\xi_4)\|_p .    \]
One can check easily check how these are all unitary representations of $\mathcal{B}_4$, for any $\xi \in \widehat{\Z}_p^4$, and their associated characters are going to have the form 
\begin{align*}
    \chi_{\pi_\xi} (\mathbf{x}) = \|(\xi_3, \xi_4)\|_p  e^{2 \pi i \{\xi \cdot \mathbf{x} \}_p} \1_{p^{- \vartheta(\xi_3, \xi_4) } \Z_p} (x_1) \int_{ \Z_p}e^{2 \pi i \{(\xi_3   x_2 +  \xi_4 x_3)u +  \xi_4 x_2 \frac{u^2}{2} \}_p}  du.
\end{align*}
Let us use these to count all the different non-equivalent representations in $B_{\widehat{\mathcal{B}}_4} (n):= \widehat{\mathcal{B}}_4 \cap \mathcal{B}_4 (p^n \Z_p)^\bot$: 

\begin{itemize}
    \item If $|\xi_4|_p=1$, then the representation is a Heisenberg representation, and for these it holds that $$\sum_{\{ \| \xi \|_p \leq p^n \, : \, |\xi_4|=1 \}} d_\xi^2 = p^{3n} = |\mathbb{H}_1 (\Z_p) / \mathbb{H}_1 (p^n \Z_p)| = |\mathbb{H}_1 (\mathbb{F}_{p^n})|.$$ 
    \item We saw before how in the case $|\xi_3|_p\leq |\xi_4|_p$ the character reduces to $$\chi_{\pi_\xi} (\mathbf{x}) =|\xi_4|_p e^{2 \pi i \{\xi_1 x_1   + \xi_2  x_2 + \xi_4 x_4 \}_p} \1_{p^{-\vartheta(\xi_4)} \Z_p} (x_1) \int_{\Z_p}e^{2 \pi i \{ \xi_4 x_3u +  \xi_4 x_2 \frac{u^2}{2} \}_p}du.$$
    So they all have dimension $|\xi_4|_p$, and we can index all the non-equivalent classes among these with the set $$\{ \| \xi \|_p \leq p^n   : 1< |\xi_4|_p\leq p^n, \, \,    |\xi_3|_p = 1, \,  \xi_1 \in \Q_p / p^{\vartheta(\xi_4)} \Z_p \}.$$
    
    \item If $1 < |\xi_4|_p < |\xi_3|_p$, the character is going to be  $$\chi_{\pi_\xi} (\mathbf{x}) =|\xi_3|_p e^{2 \pi i \{\xi_1 x_1  +  \xi_2  x_2 + \xi_3 x_3 + \xi_4 x_4 \}_p} \1_{p^{-\vartheta(\xi_3)} \Z_p} (x_1) \int_{\Z_p}e^{2 \pi i \{ \xi_4 x_3u +  \xi_4 x_2 \frac{u^2}{2} \}_p}du.$$We can index all the non-equivalent classes among these with the set $$\{ \| \xi \|_p \leq p^n   : 1< |\xi_4|_p< |\xi_3|_p , \, \,    (\xi_1 , \xi_2 , \xi_3) \in \widehat{\mathbb{H}}_1 \}.$$
\end{itemize}
In total, we can index our representations with the set
$$\widehat{\mathcal{B}}_4:= \{ \xi  \in \widehat{\Z}_p^4 \, : \, 1 \leq |\xi_4|_p <|\xi_3|_p \, \wedge \,  (\xi_1 , \xi_2 , \xi_3) \in \widehat{\mathbb{H}}_1 , \, \text{or} \,  |\xi_3|_p = 1 \,  \wedge  \,  \xi_1 \in \Q_p / p^{\vartheta(\xi_4)} \Z_p \}.$$
These are all the desired representations since

\begin{align*}
        \sum_{ \xi \in \mathcal{B}_4 \, : \, \| \xi \|_p \leq p^n } d_\xi^2 &= \sum_{ \| \xi \|_p \leq p^n \, : \, \, |\xi_4|_p =1 } d_\xi^2+ \sum_{ \| \xi \|_p \leq p^n\, :\, \, |\xi_4|_p >1, \, \, |\xi_3 |_p = 1 } d_\xi^2 +\sum_{  \| \xi \|_p \leq p^n\, :  \, \, 1<|\xi_4|_p <|\xi_3|_p, \,\, \xi_3 \in \Q_p / p^{\vartheta(\xi_4)}\Z_p} d_\xi^2 \\ &= p^{3n} + \sum_{1< |\xi_4|_p \leq p^n } \sum_{|\xi_4|_p < |\xi_3|_p \leq p^n, \,\, \xi_3 \in \Q_p / p^{\vartheta(\xi_4)}\Z_p} \sum_{(\xi_1 , \xi_2) \in \Q_p^2 / p^{\vartheta(\xi_3)} \Z_p^2}  |\xi_3|_p^2   \\ & \, \, \, \, + \sum_{1<|\xi_4|_p \leq p^n } \sum_{|\xi_2|_p \leq p^n } \sum_{\xi_1 \in \Q_p / p^{\vartheta(\xi_4)} \Z_p } |\xi_4|_p^2  \\ &=p^{3n} + \sum_{1< |\xi_4|_p \leq p^n } \sum_{|\xi_4|_p < |\xi_3|_p \leq p^n}   |\xi_3|_p^2(p^{2n} |\xi_3|_p^{-2})    + \sum_{1<|\xi_4|_p \leq p^n } |\xi_4|_p^2(p^n)(p^n|\xi_4|_p^{-1}) \\ &=p^{3n} + p^{2n} \Big( \sum_{1< |\xi_4|_p \leq p^n } \sum_{|\xi_4|_p < |\xi_3|_p \leq p^n}   1    + \sum_{1<|\xi_4|_p \leq p^n } |\xi_4|_p \Big) \\ &=p^{3n} + p^{2n} \Big( \sum_{1< |\xi_4|_p \leq p^n } (p^n - |\xi_4|_p )   + \sum_{1<|\xi_4|_p \leq p^n } |\xi_4|_p \Big) \\ &=  p^{3n} + p^{3n} \Big( \sum_{1< |\xi_4|_p \leq p^n } 1 \Big) \\ &= p^{3n} + p^{3n}(p^n -1) = p^{4n} = |\mathcal{B}_4 / \mathcal{B}_4 (p^n \Z_p)|.
\end{align*}
Finally to conclude this section, notice how from the analysis done so far, and the Fourier analysis in general compact groups, we obtain that $$f(\mathbf{x}) =\sum_{\xi \in \widehat{\mathcal{B}}_4 } \| (\xi_3 , \xi_4)\|_p  Tr[ \pi_\xi (\mathbf{x}) \widehat{f} (\xi)], \, \, \, f \in L^2 (\mathcal{B}_4),$$where $$ \widehat{f} (\xi ) \varphi = \mathcal{F}_{\mathcal{B}_4}[f](\xi) \varphi := \int_{\mathcal{B}_4} f(\mathbf{x}) \,  \pi_{\xi}^* (\mathbf{x})   \, d\mathbf{x}, \,\, \, \, \, f \in L^2 (\mathcal{B}_4) .$$In this way, the action of any densely defined left-invariant operator on $\mathcal{D}(\mathcal{B}_4)$ can be expressed as $$T f (\mathbf{x}) = \sum_{\xi \in \widehat{\mathcal{B}}_4 }\| (\xi_3 , \xi_4)\|_p  Tr[ \pi_\xi (\mathbf{x}) \sigma_T (\xi) \widehat{f} (\xi)], $$where the symbol $\sigma_T$ of the operator $T$ is the mapping $$\sigma_T: \widehat{\mathcal{B}}_4 \to \bigcup_{\xi \in \widehat{\mathcal{B}}_4} \mathcal{L}(\mathcal{H}_\xi), \, \, \, \text{defined as } \, \, \sigma_T (\xi):= \pi_\xi^* (\mathbf{x}) T \pi_\xi (\mathbf{x})  \in \mathcal{L}(\mathcal{H}_\xi) \cong \C^{\|(\xi_3, \xi_4) \|_p \times \|(\xi_3, \xi_4) \|_p}.$$
We will use all this information in the next section, where we apply it to study an important example of a left invariant operator on $\mathcal{B}_4$ which we call \emph{the Vladimirov sub-Laplacian}.

\section{The Vladimirov Sub-Laplacian on the Engel group}

We already calculated \begin{align*}
    (\pi_{\xi } (\mathbf{x}))_{h h'} &=   e^{2 \pi i \{\xi \cdot \mathbf{x} + (\xi_3   x_2 +  \xi_4 x_3)(h') +  \xi_4 x_2 \frac{(h')^2}{2} \}_p} \1_{h' - h + p^{\vartheta(\xi_3 , \xi_4)} \Z_p}(x_1) .
\end{align*}

In order to prove Theorem \ref{TeoSpectrumSublaplacianB4} we want to use this calculate the associated symbol of $\mathscr{L}_{sub}^\alpha$, and its respective invariant sub-spaces.Let us start by computing the symbols of the directional VT operators $\partial_{X_1}^{\alpha}, \partial_{X_2}^{\alpha}$,  $\alpha>0$, by using what we got for the matrix coefficients.
For $\partial_{X_2}^{\alpha}$ we have that its associated symbol $\sigma_{\partial_{X_2}^{\alpha}} (\xi) = \partial_{X_2}^{\alpha} \pi_{\xi}|_{x= 0}$ is the diagonal matrix 
\begin{align*}
    &\int_{\Z_p} \frac{\pi_\xi (0,x_2,0,0) - I_{d_\xi}}{|x_2|_p^{\alpha + 1}} dx_2 \\ &=Diag\Big(\int_{\Z_p} \frac{1}{|x_2|_p^{\alpha +1}}(e^{2 \pi i \{\xi_2 x_2  \}_p}  \Big( d_\xi \int_{h + d_\xi \Z_p}e^{2 \pi i \{\xi_3   x_2 u +  \xi_4 x_2 \frac{u^2}{2} \}_p}  du \Big) -1 )dx_2 \, : h \in \Z_p / p^{-\vartheta(\xi_3, \xi_4)} \Z_p \Big) \\ &=Diag\Big( d_\xi \int_{h + d_\xi \Z_p} \int_{\Z_p} \frac{e^{2 \pi i \{ \xi_2 x_2 + \xi_3   x_2 u +  \xi_4 x_2 \frac{u^2}{2} \}_p} - 1}{|x_2|_p^{\alpha +1}}dx_2    du  \, : h \in \Z_p / p^{-\vartheta(\xi_3, \xi_4)} \Z_p \Big)\\&=Diag\Big( \|(\xi_3 , \xi_4)\|_p \int_{h + \|(\xi_3 , \xi_4)\|_p \Z_p}  | \xi_2  + \xi_3    u +  \xi_4  \frac{u^2}{2} |_p^\alpha    du  \, : h \in \Z_p / p^{-\vartheta(\xi_3, \xi_4)} \Z_p \Big)\\&=Diag\Big( | \xi_2  + \xi_3    h +  \xi_4  \frac{h^2}{2} |_p^\alpha    \, : h \in \Z_p / p^{-\vartheta(\xi_3, \xi_4)} \Z_p \Big).
\end{align*}Now we consider the different cases:  \[ \sigma_{\partial_{X_2}^{\alpha}} (\xi)_{hh}  =\begin{cases}
 | \xi_2 |_p^\alpha - \frac{1 - p^{-1}}{1 - p^{-(\alpha +1)}}  \quad & \text{if} \quad | \xi_3 |_p = | \xi_4 |_p = 1,\\| \xi_2  + \xi_3    h  |_p^\alpha - \frac{1 - p^{-1}}{1 - p^{-(\alpha +1)}} \quad & \text{if} \quad 1= |\xi_4|_p <|\xi_3|_p, \, \, h \in \Z_p/p^{-\vartheta(\xi_3)}\Z_p, \\| \xi_2  + \xi_3    h +  \xi_4  \frac{h^2}{2} |_p^\alpha - \frac{1 - p^{-1}}{1 - p^{-(\alpha +1)}} \quad & \text{if} \quad 1< |\xi_4|_p <|\xi_3|_p, \, \, h \in \Z_p/p^{-\vartheta(\xi_3)}\Z_p,\\| \xi_2  + \xi_4  \frac{h^2}{2} |_p^\alpha - \frac{1 - p^{-1}}{1 - p^{-(\alpha +1)}} \quad & \text{if} \quad 1= |\xi_3|_p <|\xi_4|_p, \, \, h \in \Z_p/p^{-\vartheta(\xi_4)}\Z_p.
\end{cases}
\]
For $\partial_{X_1}^{\alpha}$, let us introduce some notation: let $\psi_\xi : \Z_p \to \C^{d_\xi \times d_\xi}$ be the matrix function with entries $$(\psi_{\xi})_{hh'} (u):= e^{2 \pi i \{ \xi_1 u \}_p}\1_{h' - h + p^{- \vartheta (\xi_3 , \xi_4)} \Z_p }(u), \, \, \, u \in \Z_p.$$ The associated associated symbol $\sigma_{\partial_{X_1}^{\alpha}} (\xi) = \partial_{X_1}^{\alpha} \pi_{\xi}|_{x=0}$ is a Toeplitz matrix which can be written as  \[ \sigma_{\partial_{X_1}^{\alpha}} (\xi)_{hh'}  =\begin{cases}
 | \xi_2 |_p^\alpha - \frac{1 - p^{-1}}{1 - p^{-(\alpha +1)}}  \quad & \text{if} \quad | \xi_3 |_p = | \xi_4 |_p = 1,\\\partial^\alpha_u(e^{2 \pi i \{ \xi_1 u \}_p}\1_{h' - h' + p^{- \vartheta (\xi_3 )} \Z_p })|_{u=0} \quad & \text{if} \quad |\xi_4|_p <|\xi_3|_p, \, \, h \in \Z_p/p^{- \vartheta (\xi_3 ,)}\Z_p, \\\partial^\alpha_u(e^{2 \pi i \{ \xi_1 u \}_p}\1_{h' - h + p^{- \vartheta ( \xi_4)} \Z_p })|_{u=0} \quad & \text{if} \quad 1=|\xi_3|_p< |\xi_4|_p, \, \, h \in \Z_p/p^{- \vartheta (\xi_4)}\Z_p,
\end{cases}
\]where $\partial_{u}^\alpha$ is the VT operator $$\partial_{u}^\alpha \psi (u):= \int_{\Z_p} \frac{\psi(u - t) - \psi(u) }{|t|_p^{\alpha+1}} dt, \, \, \, \, \psi \in \mathcal{D}(\Z_p).$$ Summing up, the symbol of the Vladimirov Sub-Laplacian $\mathscr{L}^\alpha_{sub}$ is:
\[ \sigma_{\mathscr{L}^\alpha_{sub}} (\xi)  =\begin{cases}
  | \xi_1 |_p^\alpha +| \xi_2 |_p^\alpha - 2\frac{1 - p^{-1}}{1 - p^{-(\alpha +1)}}  \quad & \text{if} \quad | \xi_3 |_p= |\xi_4|_p = 1, \\
(\partial_{u}^\alpha  +| \xi_2  + \xi_3    h  |_p^\alpha - \frac{1 - p^{-1}}{1 - p^{-(\alpha +1)}} )\psi_\xi (u))|_{u=0}, \quad & \text{if} \quad 1= |\xi_4|_p <|\xi_3|_p,  \\ (\partial_{u}^\alpha  +| \xi_2  + \xi_3    h +  \xi_4  \frac{h^2}{2} |_p^\alpha - \frac{1 - p^{-1}}{1 - p^{-(\alpha +1)}} )\psi_\xi (u))|_{u=0}, \quad & \text{if} \quad 1< |\xi_4|_p <|\xi_3|_p, \\ (\partial_{u}^\alpha  +| \xi_2  + \xi_4  \frac{h^2}{2} |_p^\alpha - \frac{1 - p^{-1}}{1 - p^{-(\alpha +1)}} )\psi_\xi (u))|_{u=0}, \quad & \text{if} \quad 1= |\xi_3|_p <|\xi_4|_p,
\end{cases}
\]or component-wise $\sigma_{\mathscr{L}^\alpha_{sub}} (\xi)_{hh'}  =$ \[ \begin{cases}
  | \xi_1 |_p^\alpha +| \xi_2 |_p^\alpha - 2\frac{1 - p^{-1}}{1 - p^{-(\alpha +1)}}  \quad & \text{if} \quad | \xi_3 |_p= |\xi_4|_p = 1, \\
(\partial_{u}^\alpha  +| \xi_2  + \xi_3    h  |_p^\alpha - \frac{1 - p^{-1}}{1 - p^{-(\alpha +1)}} )(e^{2 \pi i \{ \xi_1 u \}_p}\1_{h' - h + p^{- \vartheta (\xi_3 )} \Z_p }) (u))|_{u=0}, \quad & \text{if} \quad 1= |\xi_4|_p <|\xi_3|_p,  \\ (\partial_{u}^\alpha  +| \xi_2  + \xi_3    h +  \xi_4  \frac{h^2}{2} |_p^\alpha - \frac{1 - p^{-1}}{1 - p^{-(\alpha +1)}} )(e^{2 \pi i \{ \xi_1 u \}_p}\1_{h' - h + p^{- \vartheta (\xi_3 )} \Z_p }) (u))|_{u=0}, \quad & \text{if} \quad 1< |\xi_4|_p <|\xi_3|_p, \\ (\partial_{u}^\alpha  +| \xi_2  + \xi_4  \frac{h^2}{2} |_p^\alpha - \frac{1 - p^{-1}}{1 - p^{-(\alpha +1)}} )(e^{2 \pi i \{ \xi_1 u \}_p}\1_{h' - h + p^{- \vartheta (\xi_4)} \Z_p })(u))|_{u=0}, \quad & \text{if} \quad 1= |\xi_3|_p <|\xi_4|_p,
\end{cases}
\]
Let us decompose now each space $\mathcal{V}_\xi:= \mathrm{span}_\C\{(\pi_\xi(\mathbf{x}))_{hh'} \, : \, h,h' \in \Z_p / p^{- \vartheta (\xi_3 , \xi_4)} \Z_p \}$ in the direct sum $$\mathcal{V}_{\xi} = \bigoplus_{{h'} \in  \mathcal{I}_\xi} \mathcal{V}_{\xi}^{h'}, $$where 
\begin{align*}
    \mathcal{V}_{\xi}^{h'}&:= \mathrm{span}_\C \{ (\pi_{\xi})_{hh'} \, : \, h \in \Z_p/ p^{- \vartheta (\xi_3 , \xi_4)} \Z_p \}. 
\end{align*} 
Functions in this space can actually be described as $$\mathcal{V}_{\xi}^{h'} = \Big\{ e^{2 \pi i \{\xi \cdot \mathbf{x} + (\xi_3   x_2 +  \xi_4 x_3)(h') +  \xi_4 x_2 \frac{(h')^2}{2} \}_p} \varphi(x_1)  \, : \, \varphi \in \mathcal{D}_{\vartheta(\xi_3,\xi_4)}(\Z_p) \Big\},$$where we are using the notation $$\mathcal{D}_{\vartheta(\xi_3,\xi_4)}(\Z_p^d):=\{\varphi \in \mathcal{D}(\Z_p^d) \, : \, \varphi(u + v) = \varphi(u), \, \, \, | v|_p \leq \| (\xi_3, \xi_4)\|_p^{-1} \}.$$ 
For this space, we have the alternative choice of basis $$\mathscr{e}_{\xi, h', \tau}(\mathbf{x}): = e^{ 2 \pi i \{ \\xi \cdot \mathbf{x} + (\xi_3   x_2 +  \xi_4 x_3)(h') +  \xi_4 x_2 \frac{(h')^2}{2}  + \tau x_1 \}_p} , \, \, \, 1  \leq  \tau \leq \|(\xi_3 , \xi_4) \|_p,$$and for these functions we can calculate $(\partial_{X_1}^\alpha + \partial_{X_2}^\alpha) \mathscr{e}_{\xi, h', \tau}(\mathbf{x})  = $ \[\begin{cases}(|\xi_1 + \tau|_p^{\alpha} +| \xi_2  + \xi_3    h'  |_p^\alpha - 2\frac{1 - p^{-1}}{1 - p^{-(\alpha +1)}}) \mathscr{e}_{\xi, h', \tau}(\mathbf{x}) \quad & \text{if} \quad 1= |\xi_4|_p <|\xi_3|_p, \\(|\xi_1 + \tau|_p^{\alpha} +| \xi_2  + \xi_3    h'+  \xi_4  \frac{(h')^2}{2} |_p^\alpha - 2\frac{1 - p^{-1}}{1 - p^{-(\alpha +1)}}) \mathscr{e}_{\xi, h', \tau}(\mathbf{x}) \quad & \text{if} \quad 1< |\xi_4|_p <|\xi_3|_p,\\( |\xi_1 + \tau|_p^{\alpha} +| \xi_2  + \xi_4  \frac{(h')^2}{2} |_p^\alpha - 2\frac{1 - p^{-1}}{1 - p^{-(\alpha +1)}}) \mathscr{e}_{\xi, h', \tau}(\mathbf{x}) \quad & \text{if} \quad 1= |\xi_3|_p <|\xi_4|_p.
\end{cases} \]
So, the Vladimirov sub-Laplacian acts on each finite-dimensional sub-space $\mathcal{V}_{\xi}^{h'}$ like the Schr{\"o}dinger-type operator $\partial_{x_1}^\alpha  + Q(\xi, h'),$ or more precisely, in this special sub-spaces it is just a translation of the directional Vladimirov-Taibleson operator in the direction of $x_1$. Moreover, the functions $\mathscr{e}_{\xi, h', \tau}(\mathbf{x})$ form a complete system of eigenfunctions in $V_\xi^{h'}$, proving in this way how $$\| \mathscr{L}_{sub}^\alpha |_{\mathcal{V}_\xi^{h'}} \|_{op} = \max_{\tau,  h'}\lambda_{\xi, h', \tau} (\mathscr{L}_{sub}^\alpha) \lesssim  \| \xi\|_p^\alpha  .$$In the same way \begin{align*}
    \| \mathscr{L}_{sub}^\alpha |_{\mathcal{V}_\xi^{h'}} \|_{inf} &= \min_{\tau, h'}\lambda_{\xi, h', \tau} (\mathscr{L}_{sub}^\alpha) \gtrsim | \xi_1 |_p^\alpha + |\xi_2 |_p^\alpha.
\end{align*}
This implies that in each representation space we have the estimates $$| \xi_1 |_p^\alpha + |\xi_2 |_p^\alpha \lesssim \|\mathscr{L}_{sub}^\alpha |_{\mathcal{V}_\xi} \|_{inf} \leq \|\mathscr{L}_{sub}^\alpha |_{\mathcal{V}_\xi} \|_{op} \lesssim \| \xi \|_p^\alpha.,$$but we know that $$\|\mathscr{L}_{sub}^\alpha |_{\mathcal{V}_\xi} \|_{inf} = \|\sigma_{\mathscr{L}_{sub}^\alpha}(\xi)\|_{inf}, \,\,\,\|\mathscr{L}_{sub}^\alpha |_{\mathcal{V}_\xi} \|_{op} = \|\sigma_{\mathscr{L}_{sub}^\alpha}(\xi)\|_{op},$$proving how $\mathscr{L}_{sub}^\alpha$ provides an example of an operator in the class we define next:
 
\begin{defi}\label{defiellipticandsubellipticsymbol}\normalfont
Let $\sigma$ be an invariant symbol, and let $T_\sigma$ be its associated pseudo-differential operator. We say that $\sigma$ is an \emph{elliptic symbol of order} $m \in \R$, if there are $C_1, C_2 >0$ and $m \in \R$ such that $$C_1 \| \xi \|_p^m \leq \| \sigma(\xi) \|_{inf} \leq  \| \sigma(\xi)  \|_{op} \leq C_2 \| \xi \|_p^m ,$$for $\| \xi \|_p$ large enough. Similarly, if there is an $ 0 <\delta \leq m$ such that $$C_1 \| \xi \|_p^{m-\delta} \leq \| \sigma(\xi) \|_{inf} \leq  \| \sigma(\xi) \|_{op} \leq C_2 \| \xi \|_p^m ,$$for $\| \xi \|_p$ large enough, we say that $\sigma$ is a \emph{sub-elliptic symbol}. We say that an invariant operator $T_\sigma \in \mathcal{L}(H_2^{s+m}(\mathcal{B}_4) , H_2^{s}(\mathcal{B}_4))$ is elliptic or sub-elliptic, if its associated symbol is elliptic or sub-elliptic, respectively.
\end{defi}

The collection of the results proved so far constitute the statement of Theorem \ref{TeoSpectrumSublaplacianB4}, whose proof is now completed.

\nocite{*}
\bibliographystyle{acm}
\bibliography{main}
\Addresses

\end{document}